\documentclass{amsart}
\usepackage{amssymb}
\usepackage{amscd}
\usepackage{multicol}
\usepackage{tikz}
\usepackage{graphics}
\usepackage{shuffle}
\usepackage[OT2,T1]{fontenc}
\usepackage[all]{xy}
\usepackage{extarrows}
\usepackage{comment}
\usepackage{appendix}
\DeclareSymbolFont{cyrletters}{OT2}{wncyr}{m}{n}
\DeclareMathSymbol{\Sha}{\mathalpha}{cyrletters}{"58}

\usepackage{ulem}
\title[Shuffle product for multiple zeta functions ]
{Shuffle product for multiple zeta functions }
\author{Nao Komiyama}
\author{Takeshi Shinohara}
\address{Department of Mathematics, Graduate School of Science, Osaka University Toyonaka, Osaka 560-0043, Japan}
\email{komiyama.nao.aww@osaka-u.ac.jp}
\address{Graduate School of Mathematics, Nagoya University, Furo-cho, Chikusa-ku, Nagoya 464-8602 Japan}
\email{m21022w@math.nagoya-u.ac.jp}
\date{\today}
\subjclass[2020]{11M32}
\keywords{multiple zeta functions, shuffle product, hypergeometric function, zeta-functions of root systems.}

\newtheorem{thm}{Theorem}[section]
\newtheorem{lem}[thm]{Lemma}
\newtheorem{cor}[thm]{Corollary}
\newtheorem{prop}[thm]{Proposition}  

\theoremstyle{plain}
\newtheorem{ack}{Acknowledgments}

\numberwithin{equation}{section}
\theoremstyle{definition}

\newtheorem{rem}[thm]{Remark}
\newtheorem{notation}[thm]{Notation}     

\newtheorem{exa}[thm]{Examples}       
    
\newtheorem{que}[thm]{Question}   
    
\newtheorem{assume}[thm]{Assumption}

\newcommand{\re}{{\rm Re}}

\newcommand{\bfs}{{\bf s}}
\newcommand{\bft}{{\bf t}}

\newcommand{\N}{\mathbb N}
\newcommand{\Z}{\mathbb Z}
\newcommand{\Q}{\mathbb Q}
\newcommand{\R}{\mathbb R}
\newcommand{\C}{\mathbb C}

\begin{document}
\bibliographystyle{amsalpha+}

\begin{abstract}      
In this paper, we investigate the shuffle product relations for Euler-Zagier multiple zeta functions as functional relations.
To this end, we generalize the classical partial fraction decomposition formula and give two proofs.
One is based on a connection formula for Gauss's hypergeometric functions, the other one is based on an elementary calculus.
Though it is hard to write down explicit formula for the shuffle product relations for multiple zeta functions, as in the case of multiple zeta values, 
we will provide inductive steps by using the zeta-functions of root systems.
As an application, we get the functional double shuffle relations for multiple zeta functions and 
show that some relations for multiple zeta values/functions can be deduced from our results.
\end{abstract}

\maketitle
\tableofcontents
%---  Shuffle product for multiple zeta functions  ------------------------------------------------------------------------------%

%===  Section: INTRODUCTION  ===================================================================================================%
\section{Introduction}
We begin with the definition of the {\it (Euler-Zagier) multiple zeta functions}:
{\small
%---  definition: multiple zeta function  ---%
\begin{equation*}
  \zeta_r(s_1,\dots,s_r)
  :=\sum_{0<n_1<\cdots<n_r}\frac{1}{n_1^{s_1}\cdots n_r^{s_r}}
  = \sum_{m_1,\dots,m_r\in\N}\frac{1}{m_1^{s_1}(m_1+m_2)^{s_2}\cdots (m_1+\cdots+m_r)^{s_r}}
\end{equation*}
%---  definition: multiple zeta function  ---%
}
for $r\in\N$ (called {\it depth}) and complex variables $s_1,\dots,s_r\in\C$. 

Multiple zeta functions (MZFs for short) converge absolutely in:
%---  A region of absolute convergence of multiple zeta function  ---%
\begin{equation*}
  \mathcal{D}_r := \{(s_1,\dots,s_r)\in\C^r \,|\, \re(s_j+\cdots+s_r)>r-j+1,\,j=1,\dots,r \}.
\end{equation*}
%---  A region of absolute convergence of multiple zeta function  ---%
The values of MZFs at positive integer points in $\mathcal{D}_r$ are called {\it multiple zeta values} (MZVs for short).
In 1990s, it was revealed that MZVs appeared in the various areas of mathematics; 
algebraic geometry, low dimensional topology and mathematical physics, for instance 
and it has attracted many mathematicians since then.
One reason why MZVs are interesting is that they satisfy many algebraic relations over $\Q$.
For example, {\it stuffle product relations, shuffle product relations, (extended) double shuffle relations, 
sum formulas, Ohno relations, duality relations, associator relations, confluence relations, Kawashima relations, etc.} 
Here, we give two fundamental relations for MZVs;
%---  stuffle relation for the simplest case depth 1 and depth 1  ---%
\begin{align}\label{equation: stuffle product relation}
  \zeta(a)\zeta(b)
  = \zeta_2(a,b) + \zeta_2(b,a) + \zeta(a+b),
\end{align}
%---  stuffle relation for the simplest case depth 1 and depth 1  ---%

%---  shuffle relation for the simplest case depth 1 and depth 1  ---%
\begin{align}\label{equation: shuffle product relation}
  \zeta(a)\zeta(b)
  &= \sum_{k=0}^{a-1}\binom{b+k-1}{k}\zeta_2(a-k,b+k) + \sum_{k=0}^{b-1}\binom{a+k-1}{k}\zeta_2(b-k,a+k).
\end{align}
%---  shuffle relation for the simplest case depth 1 and depth 1  ---%
We call the equation \eqref{equation: stuffle product relation} (resp. equation \eqref{equation: shuffle product relation}) 
the {\it stuffle} (resp. the {\it shuffle}) product relations for the case depth $1$ and depth $1$.
As you can see, the stuffle/shuffle product relations describe that 
the product of MZVs can be expressed as a linear combination of MZVs.
By combining the above two relations, we get the non-trivial linear relations for MZVs which is called {\it (finite) double shuffle relations};
%---  double shuffle relations for the simplest case depth 1 and depth 1  ---%
\begin{align}%\label{eqn: finite double shuffle relations for double zeta values}% これ必要ないかもしれない。
  &\zeta_2(a,b) + \zeta_2(b,a) + \zeta(a+b)\notag\\
  &= \sum_{k=0}^{a-1}\binom{b+k-1}{k}\zeta_2(a-k,b+k) 
   + \sum_{k=0}^{b-1}\binom{a+k-1}{k}\zeta_2(b-k,a+k).\notag
\end{align}
%---  double shuffle relations for the simplest case depth 1 and depth 1  ---%
In general, we can obtain such relations for MZVs in a similar way. 
Furthermore, it is conjectured that all linear relations for MZVs can be deduced from {\it extended} double shuffle relations.
For the theory of MZVs, see \cite{Z16}.

From an analytic point of view, it is quite natural to ask the following question:
%---  Matsumoto-Sensei's question  ---% 
% 改めてみてもすごく自然な疑問だな。伊原先生は興味持ってくれた、とかいう記述見るけどその周辺は何か考察をしていなかったのだろうか？
\begin{que}[cf. \cite{M06}]\label{Matsumoto's question}
 {\it Algebraic relations for MZVs can be extended to functional relations for MZFs?}
\end{que} 
%---  Matsumoto-Sensei's question  ---%
This type of question was first raised by Matsumoto around 2000.
It is well known that there is a trivial answer: 
one can easily check that the stuffle product relations for MZVs (cf. \eqref{equation: stuffle product relation}) 
can be extended to functional relations for MZFs.
For the shuffle product relations, however, it is not straightforward.
Because even in the double case \eqref{equation: shuffle product relation}, 
if one considers $a\in\C$ or $b\in\C$, then the right-hand side of \eqref{equation: shuffle product relation} does not make any sense.
We encounter the same difficulty in general cases.
Many mathematicians have tried to find an answer other than stuffle product relations; 
\cite{HMO18}, \cite{HMO20}, \cite{KMT11}, \cite{MO22}, \cite{N06}, \cite{T07}, etc.

%The present paper provide one answer of Question \ref{Matsumoto's question}, namely, we focus on the shuffle product for MZVs.
% 20250227 11:00記：IPFDという手法によりMZFsの積をMZFsの無限和として表すことができることを示した。この関係式を shuffle product for MZFsと呼ぶ。
The present paper aims to extend the shuffle product relations for MZVs to functional relations for MZFs 
which might be an answer to Question \ref{Matsumoto's question}.
We shall prove that the product of two MZFs can be expressed as infinite summations in terms of MZFs by the method of 
{\it infinite partial fraction decomposition (IPFD for short, see Proposition \ref{prop: IPFD, infinite partial fraction decomposition}).}
We call those relations for MZFs {\it shuffle product relations for MZFs}, namely, our main result is;

%---  Main theorem: shuffle product relations for MZFs  --------------------------------------------------------------------------%
\bigskip
\noindent
\textbf{Theorem \ref{Main theorem}.}
\textit{Shuffle product relations hold for MZFs as functional relations.
}
\bigskip
%---  Main theorem: shuffle product relations for MZFs  --------------------------------------------------------------------------%

For details, see \S\ref{sec: General case and zeta functions of root systemsn}. 
The explicit formula for the simplest case is as follows. 

%---  Exmaple: shuffle product formula for double zeta function, explicit formula  -----------------------------------------------%
\bigskip
\noindent
\textbf{Proposition \ref{prop: explicit formula for shuffle product relation for double zeta functions}.}
\label{Example: shuffle product formula for double zeta functions}
\textit{
For $s$, $t\in\C$ with $\re(s)$, $\re(t)>1$,  
we have 
\begin{align*}
 \zeta(s)\zeta(t)
 &= \sum_{k=0}^{\infty} \left[\binom{t+k-1}{k}\zeta_2(s-k,t+k) - \binom{s+t+k-1}{s+k}\zeta_2(-k,s+t+k) \right] \\
 &\quad+ \sum_{k=0}^{\infty} \left[ \binom{s+k-1}{k}\zeta_2(t-k,s+k) - \binom{s+t+k-1}{t+k}\zeta_2(-k,s+t+k) \right].
\end{align*}
}
\bigskip
%---  Exmaple: shuffle product formula for double zeta function, explicit formula  -----------------------------------------------%

Though it is hard to write down explicit formula for general cases by using MZFs as in the case of MZVs, 
we will give the inductive steps by using the {\it zeta-functions of root systems} 
and provide some examples in \S\ref{sec: General case and zeta functions of root systemsn}.

Since we know that stuffle product relations hold for MZFs, by combining Theorem \ref{Main theorem},
we have

%---  Functional double shuffle relations  --------------------------------------------------------------------------------------%
\bigskip
\noindent
\textbf{Theorem \ref{thm: functional double shuffle product relation}}
\textit{
Functional double shuffle relations hold for MZFs.
}
\bigskip
%---  Functional double shuffle relations  --------------------------------------------------------------------------------------%

%Few examples with integer points except for all positive integers are presented in \S\ref{sec: Applications and problems}. 
This is also one answer for Question \ref{Matsumoto's question} and here we emphasize the importance of this fact.
It is easy to see that the functional double shuffle relations for MZFs with regular positive integer indices
recover the original double shuffle relations for MZVs.
In other words, the double shuffle relations for MZVs, which include many linear relations for them, 
can be extended to the functional relations for MZFs.
Therefore, we expect that most algebraic relations for MZVs can be extended to functional relations for MZFs.
 
As an application of Theorem \ref{Main theorem} and Theorem \ref{thm: functional double shuffle product relation}, 
we study the connections with previous works, 
that is, our main results can be regarded as generalizations of the following previous works.
\begin{enumerate}
\renewcommand{\labelenumi}{(\roman{enumi})}
\item 
Functional relations for the zeta-functions of root systems including the shuffle product relations for MZVs studied in \cite{KMT11} 
(see Corollary \ref{KMT's results}).
\item  
{\it Sum formulas and extended double shuffle relations for MZFs} 
(see Corollary \ref{thm: functional sum formula} and Corollary \ref{cor: EDSR for MZFs}).
\end{enumerate}

The plan of the present paper is as follows:
in \S\ref{sec: Infinite partial fraction decomposition}, 
we generalize the classical partial fraction decomposition formula 
by using the connection formula for Gauss's hypergeometric functions or an elementary calculus. 
We also give a proof of the shuffle product relations for double zeta functions.
In \S\ref{sec: General case and zeta functions of root systemsn}, 
we prove our main theorem.
We also exposit how to calculate the shuffle product relations for MZFs by using the zeta-functions of root systems.
Some explicit examples will be presented here.
In \S\ref{sec: Applications and problems},
we show that some previous works can be deduced from our results and discuss some miscellaneous topics.
In \S\ref{sec:Concluding remarks}, 
we will explore possible connections between the recurrence relations for MZFs and the shuffle product relations for MZFs under
certain assumptions.

%---  Shuffle product for multiple zeta function  -------------------------------------------------------------------------------% 

{\ack {\rm The authors would like to thank Kohji Matsumoto for his helpful comments, 
Takashi Nakamura for a careful reading of the manuscript, and Keita Nakai for his valuable comments.

N.K. has been supported by grants JSPS KAKENHI JP23KJ1420.
T.S. has been supported by grants JSPS KAKENHI JP24KJ1252. }}

%===  SECTION END  ==============================================================================================================%

%===  Section: INFINITE PARTIAL FRACTION DECOMPOSITION  =========================================================================%
\section{Infinite partial fraction decomposition}\label{sec: Infinite partial fraction decomposition}
In this section, we shall prove a key formula (Proposition \ref{prop: IPFD, infinite partial fraction decomposition}). 
It is known that the shuffle product relations for MZVs can be proved by
using an {\it iterated integral expression} or the method of {\it partial fraction decomposition (PFD for short).}
For $(k_1,\dots,k_r)\in\N$ with $k_r\ge2$, multiple zeta value $\zeta_r(k_1,\dots,k_r)$ has the iterated integral expression;
\begin{align*}
  \zeta_r(k_1,\dots,k_r) 
  = I(\underbrace{0,\dots,0}_{k_1-1},1,\dots,\underbrace{0,\dots,0}_{k_r-1},1),
  % \\ = \int_{0<t<1}\underbrace{k_r-1}\omega_0(t)\circ\omega_1(t)\circ\cdots\circ\underbrace{k_r-1}\omega_0(t)\circ\omega_1(t)
\end{align*} 
where
\begin{align*}
 I(\epsilon_1,\dots,\epsilon_n)
 := \int_{0<t_1<\cdots<t_n<1}\prod_{j=1}^{n}\omega_{\epsilon_j}(t_j) 
\end{align*}
for $\epsilon_j\in\{0,\,1\}$ and $\omega_0(t)=\frac{dt}{t}$, $\omega_1(t)=\frac{dt}{1-t}$.
This expression yields the shuffle product relations for MZVs.
One sees that the iterated integral expression of MZVs is valid only when $k_1,\dots,k_r$ are positive integers ($k_r\ge2$)
and therefore we have to find another way for our interest.

%---  Remark: Joyner's general iterated integral  ---%
\begin{rem}
Joyner (\cite{J10}) introduced the iterated integrals for complex variables via the Mellin transform 
and studied the connection with MZFs, multiple polylogarithms, etc.
However, it seems that he does not give the (shuffle) product formula for them.
At least authors could not find explicit formula for the product of general iterated integrals in \cite{J10}.
\end{rem}
%---  Remark: Joyner's general iterated integral  ---%

We next explain the method of PFD.
The following formula is fundamental; 
%\footnote{Historically, Euler (\cite{E1776}) had already used this formula to prove some relations for double zeta values.}
for $a$, $b\in\N$,
\begin{align}\label{eqn: classical partial fractin decompositon}
 \frac{1}{x^ay^b} 
 = \sum_{k=0}^{a-1}\binom{b+k-1}{k} \frac{1}{x^{a-k}(x+y)^{b+k}} + \sum_{k=0}^{b-1}\binom{a+k-1}{k} \frac{1}{y^{b-k}(x+y)^{a+k}}. 
\end{align}
We remark that if we take summation over all $x$, $y\in\N$ on both sides of \eqref{eqn: classical partial fractin decompositon},
then we get the shuffle product relations for double zeta values \eqref{equation: shuffle product relation}.
Additionally, Komori, Matsumoto and Tsumura (\cite{KMT11}) explained how to apply \eqref{eqn: classical partial fractin decompositon} 
to the product of MZVs for general depths 
and proved the shuffle product relations for MZVs without the iterated integral expression.
Furthermore, they succeeded in providing certain functional relations for MZFs and zeta-functions of root systems 
that include the shuffle product relations for MZVs, see \S\ref{sec: Applications and problems}.

The next lemma plays an important role in our paper.
%---  Lemma: a certain equation which comes from the connection formula of 2F1  --------------------------------------------------%
\begin{lem}\label{lem: pieces of shuffle product relation for dzf}
For $s$, $t$, $x$, $y\in\C$ with $\re(s)$, $\re(t)$, $\re(x)$, $\re(y)>0$, we have 
\begin{equation*}%\label{eqn: pieces of shuffle product relation for dzf}
 \frac{(x+y)^{s+t}}{x^sy^t} 
 = \frac{\Gamma(s+t)}{\Gamma(s)\Gamma(t)}
    \left(\frac{1}{s}\,\cdot {}_2F_1\left(\begin{matrix} s+t,\, 1 \\ s+1 \end{matrix}\,\middle|\, \frac{x}{x+y}\right) 
    + \frac{1}{t}\, \cdot{}_2F_1\left(\begin{matrix} s+t,\, 1 \\ t+1 \end{matrix}\,\middle|\, \frac{y}{x+y}\right)
    \right),
\end{equation*}
where ${}_2F_1\left(\begin{matrix} \alpha,\, \beta \\ \gamma \end{matrix}\,\middle|\, z\right)$ is the Gauss's hypergeometric function;
\begin{align*}
  {}_2F_1\left(\begin{matrix} \alpha,\, \beta \\ \gamma \end{matrix}\,\middle|\, z\right)
  := \sum_{n=0}^{\infty} \frac{(\alpha)_n(\beta)_n}{(\gamma)_n n!}z^n
\end{align*}
for $z\in\C$ with $|z|<1$, $\alpha$, $\beta$, $\gamma\in\C$ with $\gamma\notin\Z_{\le0}$, 
$(X)_0=1$, $(X)_n=X(X+1)\cdots (X+n-1)$ are the Pochhammer symbols and
$\Gamma(s)$ is the gamma function.
\end{lem}    
%---  Lemma: a certain equation which comes from the connection formula of 2F1  -------------------------------------------------%

%---  Proof: a certain equation which comes from the connection formula of 2F1  -------------------------------------------------%
\begin{proof} 
We first review the connection formula for hypergeometric function ${}_2F_1$, see \cite[\S14.53]{WW27} for details.
For $z\in\C$, $\alpha$, $\beta$, $\gamma\in\C$ 
with $\alpha$, $\beta$, $\gamma-\alpha$, $\gamma-\beta$, $\pm(\alpha+\beta-\gamma)+1\notin\Z_{\le0}$, we have
\begin{align*}
    {}_2F_1\left(\begin{matrix} \alpha,\, \beta \\ \gamma \end{matrix}\middle|\, z \right)
    &= \frac{\Gamma(\gamma)\Gamma(\gamma-\alpha-\beta)}{\Gamma(\gamma-\alpha)\Gamma(\gamma-\beta)}
       {}_2F_1\left(\begin{matrix} \alpha,\, \beta \\ \alpha+\beta-\gamma+1 \end{matrix}\,\middle|\, 1-z \right)\\
     &\quad+\frac{\Gamma(\gamma)\Gamma(\alpha+\beta-\gamma)}{\Gamma(\alpha)\Gamma(\beta)}
       (1-z)^{\gamma-\alpha-\beta}
       {}_2F_1\left(\begin{matrix} \gamma-\alpha,\, \gamma-\beta \\ \gamma-\alpha-\beta+1 \end{matrix}\,\middle|\, 1-z \right).
\end{align*}
By putting $\alpha=s+t$, $\beta=1$, $\gamma=s+1$, $z=\frac{x}{x+y}$ in the above formula, one has
\begin{align*}
    {}_2F_1\left(\begin{matrix} s+t,\, 1 \\ s+1 \end{matrix}\,\middle|\, \frac{x}{x+y} \right)
    &= \frac{\Gamma(s+1)\Gamma(-t)}{\Gamma(1-t)\Gamma(s)} 
       {}_2F_1\left(\begin{matrix} s+t,\, 1 \\ t+1 \end{matrix}\,\middle|\, \frac{y}{x+y} \right)\\
    &\quad + \frac{\Gamma(s+1)\Gamma(t)}{\Gamma(s+t)\Gamma(1)}\left(\frac{y}{x+y}\right)^{-t}
       {}_2F_1\left(\begin{matrix} 1-t,\, s \\ 1-t \end{matrix}\,\middle|\, \frac{y}{x+y} \right)\\
    &= -\frac{s}{t}\, {}_2F_1\left(\begin{matrix} s+t,\, 1 \\ t+1 \end{matrix}\,\middle|\, \frac{y}{x+y} \right)
       + s\frac{\Gamma(s)\Gamma(t)}{\Gamma(s+t)} \frac{(x+y)^{s+t}}{x^sy^t}.
\end{align*}
This implies the lemma.
\end{proof}
%---  Proof: a certain equation which comes from the connection formula of 2F1  --------------------------------------------------%

We generalize the classical PFD \eqref{eqn: classical partial fractin decompositon}.
For our simplicity, we put
\begin{align}\label{eqn:general binomial coefficient}
\binom{s}{t}:=\frac{\Gamma(s+1)}{\Gamma(t+1)\Gamma(s-t+1)}
\end{align}
for $s,t\in\C$.
When $s,t\in\N$ with $s\ge t$, $\binom{s}{t}$ is nothing but the original binomial coefficient.

%---  Proposition: Infinite PFD  -------------------------------------------------------------------------------------------------%
\begin{prop}[Infinite PFD]\label{prop: IPFD, infinite partial fraction decomposition}
For $s$, $t$, $x$, $y\in\C$ with $\re(s)$, $\re(t)$, $\re(x)$, $\re(y)>0$,
\begin{align*}\label{eqn: IPFD, infinite partial fraction decomposition}\tag{IPFD}
 \frac{1}{x^sy^t} 
 &= \sum_{k=0}^{\infty} 
     \left[ \binom{t+k-1}{k}\frac{1}{x^{s-k}(x+y)^{t+k}} - \binom{s+t+k-1}{s+k}\frac{1}{x^{-k}(x+y)^{s+t+k}} \right]\\
 &\quad+ \sum_{k=0}^{\infty} 
          \left[ \binom{s+k-1}{k}\frac{1}{y^{t-k}(x+y)^{s+k}} - \binom{s+t+k-1}{t+k}\frac{1}{y^{-k}(x+y)^{s+t+k}} \right].
\end{align*}
\end{prop}
%---  Proposition: Infinite PFD  -------------------------------------------------------------------------------------------------%

%---  Proof: Infinite PFD  -------------------------------------------------------------------------------------------------------%
\begin{proof}
By Lemma \ref{lem: pieces of shuffle product relation for dzf}, we have 
\begin{align*}
  \frac{1}{x^sy^t} 
  &= \frac{1}{x^sy^t} - \frac{\Gamma(s+t)}{\Gamma(s)\Gamma(t)}\ 
                         \frac{1}{s}\,\cdot {}_2F_1\left(\begin{matrix} s+t,\, 1 \\ s+1 \end{matrix}\,\middle|\, \frac{x}{x+y}\right)
                         \frac{1}{(x+y)^{s+t}} \\
  &\quad+ \frac{1}{x^sy^t} - \frac{\Gamma(s+t)}{\Gamma(s)\Gamma(t)}\ 
                         \frac{1}{t}\, \cdot{}_2F_1\left(\begin{matrix} s+t,\, 1 \\ t+1 \end{matrix}\,\middle|\, \frac{y}{x+y}\right)
                         \frac{1}{(x+y)^{s+t}}.
\end{align*}
Thus, we immediately get the claim by the generalized binomial theorem:
$$
(1+z)^\alpha=\sum_{k=0}^\infty\binom{\alpha}{k}z^k \quad (z\in\C,\,|z|<1).
$$
\end{proof}
%---  Proof: Infinite PFD  ------------------------------------------------------------------------------------------------------%

Note that this equation recovers the PFD \eqref{eqn: classical partial fractin decompositon} 
by putting $s$, $t\in\N$.
Indeed, if $s$, $t\in\N$, telescoping happens on the right-hand side. 

%---  Remark: Another proof of Infinite PFD by elementary calculus  -------------------------------------------------------------%
\begin{rem}
We can prove our key equation \eqref{eqn: IPFD, infinite partial fraction decomposition} elementary as follows.

%---  Proof: easy calculation of Mellin integral formula  -----------------------------------------------------------------------%
\noindent
{\it Another proof of equation \eqref{eqn: IPFD, infinite partial fraction decomposition}}:
By the Mellin transform, we have
\begin{align*}
    \frac{1}{x^sy^t}
    &= \int_{[0,\infty)^2}\frac{u^{s-1}v^{t-1}}{\Gamma(s)\Gamma(t)}e^{-xu}e^{-yv}dudv\\
    &= \int_{0<u<v} \frac{u^{s-1}v^{t-1}}{\Gamma(s)\Gamma(t)}e^{-xu}e^{-yv}dudv 
      +\int_{0< v< u} \frac{u^{s-1}v^{t-1}}{\Gamma(s)\Gamma(t)}e^{-xu}e^{-yv}dudv.
\end{align*}
Due to the symmetry, we now compute the first term on the right-hand side of the last equality. 
By integration by parts, we have 
\begin{align*}
    &\int_{0<u<v}\frac{u^{s-1}v^{t-1}}{\Gamma(s)\Gamma(t)}e^{-xu}e^{-yv}dudv \\
    %= \int_{0}^{\infty}\frac{t^{a-1}}{\Gamma(a)}e^{-xt}\left(\int_{t}^{\infty}\frac{s^{b-1}}{\Gamma(b)}e^{-ys}ds\right)\, dt\\
    %&= \int_{0}^{\infty}\frac{t^{a-1}}{\Gamma(a)}e^{-xt}
    %    \left( \left[ \frac{s^{b}e^{-ys}}{b\Gamma(b)} \right]_t^{\infty} 
    %    + y\int_{t}^{\infty}\frac{s^{b}}{\Gamma(b+1)}e^{-ys}ds \right)dt\\
    &= \int_{0}^{\infty}\frac{u^{s-1}}{\Gamma(s)}e^{-xu}
        \left( -\frac{u^{t}e^{-yu}}{\Gamma(t+1)} + y\int_{u}^{\infty}\frac{v^{t}}{\Gamma(t+1)}e^{-yv}dv \right)du\\
    &= \int_{0}^{\infty}\frac{u^{s-1}}{\Gamma(s)}e^{-xu}
        \left( 
            -\sum_{k=0}^{R-1}
              \frac{u^{t+k}e^{-yu}}{\Gamma(t+k+1)}y^k 
               + y^{R}\int_{u}^{\infty} \frac{v^{t+R-1}}{\Gamma(t+R)}e^{-yv}dv 
        \right)du\\
    &= \int_{0}^{\infty}\frac{u^{s-1}}{\Gamma(s)}e^{-xu}
        \left(
            -\sum_{k=0}^{\infty}\frac{u^{t+k}e^{-yu}}{\Gamma(t+k+1)}y^k + y^{-t} 
        \right)du\\
    %&= y^{-b}\int_{0}^{\infty}\frac{t^{a-1}}{\Gamma(a)}e^{-xt}
    %   -\sum_{k=0}^{\infty} \frac{\Gamma(a+b+k)}{\Gamma(a)\Gamma(b+k+1)}
    %                        y^k\int_{0}^{\infty}\frac{t^{a+b+k-1}}{\Gamma(a+b+k)}e^{-(x+y)t}dt \\
    &= \frac{1}{x^sy^t} 
       - \sum_{k=0}^{\infty} \binom{s+t+k-1}{t+k}\frac{1}{y^{-k}(x+y)^{s+t+k}}.
\end{align*}
%We repeated $R$ times  in the $4$-th equality.
The third equality holds because of Lemma \ref{lem: limit of ratio of the upper incomplete gamma and gamma functions}. 
Therefore, we obtain the equation \eqref{eqn: IPFD, infinite partial fraction decomposition}
by applying the generalized binomial theorem to the first term on the right-hand side of the last equality. \qed
\end{rem}
%---  Remark: Another proof of Infinite PFD by elementary calculus  -------------------------------------------------------------%

%---  Lemma: limit of ratio of the upper incomplete gamma and gamma functions  ----------------------------------------------%
\begin{lem}\label{lem: limit of ratio of the upper incomplete gamma and gamma functions}
For $x\in\R_{\ge0}$, we have
\[
  \frac{\Gamma(R,x)}{\Gamma(R)} \rightarrow 1 \quad (R\rightarrow\infty),
\]
where, $\Gamma(s,x)$ is the upper incomplete gamma function; for $s\in\C$, $x\in\R_{\ge0}$, 
\[
  \Gamma(s,x) := \int_{x}^{\infty}t^{s-1}e^{-t}dt.
\]
\end{lem}
%---  Lemma: limit of ratio of the upper incomplete gamma and gamma functions  -----------------------------------------------%

%---  Proof of lemma limit of ratio of the upper incomplete gamma and gamma functions  ---------------------------------------%
\begin{proof}
  It is easy to check that the claim is valid.
\end{proof}
%---  Proof of lemma limit of ratio of the upper incomplete gamma and gamma functions  ---------------------------------------%

By taking the summation over all $x$, $y\in\N$ on both sides of \eqref{eqn: IPFD, infinite partial fraction decomposition},
one may expect that one can obtain shuffle product relations for double zeta functions 
\eqref{eqn: shuffle product relation for double zeta functions}.
Indeed, we are able to do that because the next lemma justifies the change of order of infinite summations.

%---  Lemma: change of order of two infinite sums  -----------------------------------------------------------------------------%
\begin{lem}\label{lem: justification of the change of order of two infinite summations}
For $s$, $t\in\C$ with $\re(s)$, $\re(t)>1$,  the following series absolutely converges:
\[
  \sum_{k=0}^{\infty}\left[ \binom{t+k-1}{k}\zeta_2(s-k,t+k) - \binom{s+t+k-1}{s+k}\zeta_2(-k,s+t+k) \right].
\]
\end{lem}
%---  Lemma: change of order of two infinite sums  -----------------------------------------------------------------------------%

%---  Proof: change of order of two infinite sums  ------------------------------------------------------------------------------%
\begin{proof}
We just evaluate the summand. For sufficient large $k$, we have
\begin{align*}
  &\left| \binom{t+k-1}{k}\zeta_2(s-k,t+k) - \binom{s+t+k-1}{s+k}\zeta_2(-k,s+t+k) \right| \\
  &\le \sum_{m,n\in\N}\left| \binom{t+k-1}{k}\frac{1}{m^{s-k}(m+n)^{t+k}} - \binom{s+t+k-1}{s+k}\frac{1}{m^{-k}(m+n)^{s+t+k}} \right| \\
  &= \sum_{m,n\in\N} \left|\frac{1}{\Gamma(t)} \frac{1}{(m+n)^t}\left(\frac{m}{m+n}\right)^k\right| \cdot 
      \left|
       \frac{\Gamma(t+k)}{\Gamma(k+1)}\frac{1}{m^s} - \frac{\Gamma(s+t+k)}{\Gamma(s+k+1)}\frac{1}{(m+n)^s}
      \right|\\
  &= \sum_{m,n\in\N} 
      \left|\frac{1}{\Gamma(t)} \frac{1}{(m+n)^t}\left(\frac{m}{m+n}\right)^kk^{t-1}(1+o(1))\right| \cdot 
      \left|
       \frac{1}{m^s} - \frac{1}{(m+n)^s} 
      \right|.
\end{align*}
In the last equality, we use the result by Tricomi and Erd\'elyi \cite{TE};
\begin{align*}
\frac{\Gamma(x+a)}{\Gamma(x+b)}
 = x^{a-b}\left( 1 + \frac{(a-b)(a+b-1)}{2x} + O(x^{-2}) \right) \quad (\text{as } x\rightarrow\infty,\,a,\,b\in\C).
\end{align*}
By the binomial theorem, we have
\[
   \left| \left(\frac{m}{m+n}\right)^k k^{t-1} \right|
   = \left| \frac{k^{t-1}}{(1+n/m)^k} \right|
   = \left(\frac{m}{n}\right)^{\lfloor \re(t)\rfloor+1 }O\left(\frac{1}{k^{2-\{\re(t)\}}}\right).
\]
For $x\in\R$, the symbol $\lfloor x\rfloor$ (resp. $\{x\}$) stands for the integer (resp. fractional) part of $x$.
Therefore one obtain the claim if the series
\begin{equation}\label{eqn: mtmzf}
  \sum_{m,n\in\N}
   \left| \frac{1}{m^{-\lfloor \re(t)\rfloor -1}n^{\lfloor \re(t)\rfloor+1}(m+n)^{t}} \right|
   \left| \frac{1}{m^s} - \frac{1}{(m+n)^s} \right| 
\end{equation}
converges absolutely.
Note that 
\begin{align*}
 \left| \frac{1}{m^s} - \frac{1}{(m+n)^s} \right|
 &= \left| \int_{m+n}^m(-s)u^{-s-1}du \right| \\
 &\le |s|\int_{m}^{m+n}u^{-\re(s)-1}du \\
 &= \frac{|s|}{\re(s)}\left( \frac{1}{m^{\re(s)}} - \frac{1}{(m+n)^{\re(s)}} \right).
\end{align*}
For $\re(s)>2$, since $|(\frac{m}{m+n})^{\re(s)}|<1$, we have
{\small
\begin{align*}
   &\sum_{m,n\in\N}\left|\frac{1}{m^{\re(s)-\lfloor \re(t)\rfloor -1}n^{\lfloor \re(t)\rfloor+1}(m+n)^t}
    \left(1 - \left(\frac{m}{m+n}\right)^{\re(s)} \right)\right| \\
   &\le \sum_{m,n\in\N}\left|\frac{1}{m^{\re(s)-\lfloor \re(t)\rfloor -1}n^{\lfloor \re(t)\rfloor+1}(m+n)^t}\right|
   = \sum_{m,n\in\N}\frac{1}{m^{\re(s)-\lfloor \re(t)\rfloor -1}n^{\lfloor \re(t)\rfloor+1}(m+n)^{\re(t)}}.
\end{align*}
}
The right-hand side of the above inequality is nothing but the Mordell-Tornheim double zeta function
\footnote{Mordell-Tornheim double zeta function $\sum_{m,n\in\N}m^{-r}n^{-s}(m+n)^{-t}$ 
converges when $\re(r+t)>1$, $\re(s+t)>1$ and $\re(r+s+t)>2$.} 
and one can easily check that it converges.
Similarly, for $1< \re(s)\le2$, 
since $|1-(\frac{m}{m+n})^{\re(s)}|\le|1-(\frac{m}{m+n})^2|=\frac{2mn}{(m+n)^2}+\frac{n^2}{(m+n)^2}$,
one can show that the series \eqref{eqn: mtmzf} converges.
\end{proof}
%---  Proof: change of order of two infinite sums  ------------------------------------------------------------------------------%

By combining Proposition \ref{prop: IPFD, infinite partial fraction decomposition} 
and Lemma \ref{lem: justification of the change of order of two infinite summations}, 
we get the shuffle product relations for double zeta functions.

%---  Proposition: explicit formula for shuffle product relation for double zeta functions  -------------------------------------% 
\begin{prop}\label{prop: explicit formula for shuffle product relation for double zeta functions}
For $s$, $t\in\C$ with $\re(s)$, $\re(t)>1$, we have
\begin{align}\label{eqn: shuffle product relation for double zeta functions}
 \zeta(s)\zeta(t)
  &= \sum_{k=0}^{\infty} \left[\binom{t+k-1}{k}\zeta_2(s-k,t+k) - \binom{s+t+k-1}{s+k}\zeta_2(-k,s+t+k) \right] \\
  &\quad+ \sum_{k=0}^{\infty} \left[ \binom{s+k-1}{k}\zeta_2(t-k,s+k) - \binom{s+t+k-1}{t+k}\zeta_2(-k,s+t+k) \right]\notag.
\end{align}
\end{prop}
%---  Proposition: explicit formula for shuffle product relation for double zeta functions  -------------------------------------%

%---  Proof of proposition: explicit formula for shuffle product relation for double zeta functions  ----------------------------%
\begin{proof}
This is a direct consequence of Proposition \ref{prop: IPFD, infinite partial fraction decomposition} and
Lemma \ref{lem: justification of the change of order of two infinite summations}.
\end{proof}
%---  Proof of proposition: explicit formula for shuffle product relation for double zeta functions  ----------------------------%

We remark that the equation \eqref{eqn: shuffle product relation for double zeta functions}
recovers the original shuffle product relations \eqref{equation: shuffle product relation} by setting $s$, $t\in\N_{\ge2}$.
Moreover, by applying stuffle product relations, we get the {\it functional double shuffle relations} for double zeta functions.

%---  Corollary: Functional double shuffle relations  ----------------------------------------------------------------------------%
\begin{cor}\label{cor: functional double shuffle relations for double}
For $s$, $t\in\C$ with $\re(s)$, $\re(t)>1$, we have
\begin{align*}
  &\zeta_2(s,t) + \zeta_2(t,s) + \zeta(s+t) \\
  &= \sum_{k=0}^{\infty} \left[\binom{t+k-1}{k}\zeta_2(s-k,t+k) - \binom{s+t+k-1}{s+k}\zeta_2(-k,s+t+k) \right]\notag \\
  &\quad+ \sum_{k=0}^{\infty} \left[ \binom{s+k-1}{k}\zeta_2(t-k,s+k) - \binom{s+t+k-1}{t+k}\zeta_2(-k,s+t+k) \right]\notag.
\end{align*}
\end{cor}
%---  Corollary: Functional double shuffle relations  --------------------------------------------------------------------------%
%===  SECTION END  =============================================================================================================%  

%===  Section: GENERALIZATION  ==================================================================================================%
\section{General case and zeta functions of root systems}\label{sec: General case and zeta functions of root systemsn}
In this section, we first show our main theorem which is a generalization of
Proposition \ref{prop: explicit formula for shuffle product relation for double zeta functions} to higher depths.
And then, we will exposit the inductive steps 
by using the zeta-functions of root systems and introducing a new way to write variables of those functions 
in Notation \ref{Notation: matrices like variavles of zeta functions of root systems of type Ar}.
This notation makes it easier to understand the process of the shuffle product
and will be seen in Example \ref{example: depth 2 times 1, 3 times 1, and 2 times 2}.

%---  subsection: Main results  -------------------------------------------------------------------------------------------------%
\subsection{Main result}\label{subsec:Main results}
By applying IPFD to the product of summands of MZFs inductively and taking summation, we obtain the functional relation for MZFs. 
We call such relations for MZFs, that is, functional relations for MZFs obtained by IPFD, {\it shuffle product relations for MZFs}.
In this sense, we state our main result.
%---  Main theorem: shuffle product relations for MZFs  -------------------------------------------------------------------------%
\begin{thm}\label{Main theorem}
Shuffle product relations hold for MZFs as functional relations.
Namely, by IPFD, the product of two MZFs $\zeta_p(s_1,\dots,s_p)\zeta_q(t_1,\dots,t_q)$ 
can be expressed as infinite summations in terms of MZFs (depth $p+q$) 
for $s_1$,$\dots$,$s_p$, $t_1$,$\dots$,$t_q\in\C$ with 
\begin{align*}
 \re(s_p),\,\re(t_q)>1,\  \re(s_i),\,\re(t_j)\ge1 \quad (1\le i\le p-1,\  1\le j \le q-1).
\end{align*}
\end{thm}
%---  Main theorem: shuffle product relations for MZFs  --------------------------------------------------------------------------%
  
%---  Proof of theorem: shuffle product relations for MZFs  ----------------------------------------------------------------------%
\begin{proof}
We use the following notation; for $p$, $q\in\N$, $m_1,\dots,m_p$, $n_1,\dots,n_q\in\N$, 
$(s_1,\dots,s_p)\in\C^p$ and $(t_1,\dots,t_q)\in \C^q$, 
set $M_j:=\sum_{i=1}^jm_i$ ($1\le j\le p$), $N_j:=\sum_{i=1}^jn_i$ ($1\le j\le q$) and
\begin{align*}
 \mathfrak{z}_p\left( \begin{matrix} s_1,\dots,s_p \\ m_1,\dots,m_p \end{matrix} \right)
  := \frac{1}{M_1^{s_1}\cdots M_p^{s_p}}.
\end{align*}
We first show that following claim;
  
\noindent
\textbf{Claim.}
{\it
For $p$, $q\in\N$, the product 
\begin{align*}
 \mathfrak{z}_p\left( \begin{matrix} s_1,\dots,s_p \\ m_1,\dots,m_p \end{matrix} \right)
  \mathfrak{z}_q\left( \begin{matrix} t_1,\dots,t_q \\ n_1,\dots,n_q \end{matrix} \right)
\end{align*}  
can be expressed as infinite summations in terms of $\mathfrak{z}_{p+q}$.
}
  
We prove this claim by induction on $p+q$.
We have already proved the case $p+q=2$.
Assume $p+q>2$.
By \eqref{eqn: IPFD, infinite partial fraction decomposition},
\begin{align*}
 &\mathfrak{z}_p\left( \begin{matrix} s_1,\dots,s_p \\ m_1,\dots,m_p \end{matrix} \right)
  \mathfrak{z}_q\left( \begin{matrix} t_1,\dots,t_q \\ n_1,\dots,n_q \end{matrix} \right) \\
  %\frac{1}{N_1^{t_1}\cdots N_q^{t_q}} \\
 &= \sum_{k=0}^{\infty}
     \left[ 
      \binom{t_q+k-1}{k}
       \mathfrak{z}_p\left( \begin{matrix} s_1,\dots,s_p-k \\ m_1,\dots,m_p \end{matrix} \right)
        \mathfrak{z}_{q-1}\left( \begin{matrix} t_1,\dots,t_{q-1} \\ n_1,\dots,n_{q-1} \end{matrix} \right) 
         \frac{1}{(M_p+N_q)^{t_q+k}} \right.\\
 &\quad\left.- \binom{s_p+t_q+k-1}{s_p+k}
        \mathfrak{z}_{p}\left( \begin{matrix} s_1,\dots,s_{p-1},-k \\ m_1,\dots,m_p \end{matrix} \right)
         \mathfrak{z}_{q-1}\left( \begin{matrix} t_1,\dots,t_{q-1} \\ n_1,\dots,n_{q-1} \end{matrix} \right)
          \frac{1}{(M_p+N_q)^{s_p+t_q+k}} 
       \right] \\
 &+ \sum_{k=0}^{\infty}
     \left[ 
      \binom{s_p+k-1}{k}
       \mathfrak{z}_{p-1}\left( \begin{matrix} s_1,\dots,s_{p-1} \\ m_1,\dots,m_{p-1} \end{matrix} \right)
        \mathfrak{z}_q\left( \begin{matrix} t_1,\dots,t_{q}-k \\ n_1,\dots,n_q \end{matrix} \right) 
         \frac{1}{(M_p+N_q)^{s_p+k}} \right.\\
 &\quad\left.
     -\binom{s_p+t_q+k-1}{t_q+k}
       \mathfrak{z}_{p-1}\left( \begin{matrix} s_1,\dots,s_{p-1} \\ m_1,\dots,m_{p-1} \end{matrix} \right)
        \mathfrak{z}_{q}\left( \begin{matrix} t_1,\dots,t_{q-1},-k \\ n_1,\dots,n_{q} \end{matrix} \right)
         \frac{1}{(M_p+N_q)^{s_p+t_q+k}}
       \right].
\end{align*}
We use the induction hypothesis on the right-hand side and get the claim.

Moreover, by taking the summation over all $m_1,\dots,m_p, n_1,\dots,n_q\in\N$, we get the theorem. 
The convergence can be shown as in Lemma \ref{lem: justification of the change of order of two infinite summations}.
\end{proof}
%---  Proof of theorem: shuffle product relations for MZFs  ---------------------------------------------------------------------%
%---  subsection: Main results  -------------------------------------------------------------------------------------------------%
%===  Section: GENERALIZATION  ==================================================================================================%

%===  Section: GENERALIZATION  ==================================================================================================%
%---  subsection: zeta functions of root systems  -------------------------------------------------------------------------------%
\subsection{Zeta-functions of root systems}\label{subsec:zeta functions of root systems}
To describe the inductive steps of the shuffle product relations for MZFs, 
we use the zeta-functions of root systems which are the generalization of various multiple zeta functions; 
Euler-Zagier type, Mordell-Tornheim type, Apostol-Vu type and Witten zeta functions.
They also have connections with various areas of mathematics; mathematical physics and representation theory, for instance.
However, we use only zeta-functions of root systems of type $A_r$, so we omit the precise definition here. 
For the theory of the zeta-functions of root systems, see \cite{KMT24}.
The zeta-functions of root systems of type $A_r$ have the following explicit form;
%---  Proposition: From KMT 24, Explicit formula for zeta function of root systems of type Ar  ----------------------------------%
\begin{prop}[{cf. \cite[Proposition 3.5]{KMT24}}]\label{prop: Explicit formula for zeta function of root systems of type Ar}
For $r\in\N$ and $\bfs=(s_{ij})\in\C^{r(r+1)/2}$, we have
\begin{align*}%\label{equation: the zeta function of root systems of type Ar}
 \zeta_r(\bfs;A_r)
  = \sum_{m_1,\dots,m_r\in\N} \prod_{1\le i\le j\le r}(m_i+\cdots+m_j)^{-s_{ij}}. 
\end{align*}
\end{prop}
%---  Proposition: From KMT 24, Explicit formula for zeta function of root systems of type Ar  -----------------------------------%

In the present paper, we regard the above equation 
as the definition of the zeta-functions of root systems of type $A_r$. 
It is easy to check that the series on the right-hand side converges absolutely in 
$\{(s_{ij})\in\C^{r(r+1)/2}\,|\,\re(s_{ij})>1,\,  1\le i\le j\le r\}$.

%---  Notation: Matrices like variavles of zeta functions of root systems of type Ar  --------------------------------------------%
\begin{notation}\label{Notation: matrices like variavles of zeta functions of root systems of type Ar}
In this paper, we occasionally use an alternative expression of variables of $\zeta_r(\bfs;A_r)$.
\begin{align} \label{eqn: figure of variables of zeta functions of type Ar}
  \zeta_r(\bfs;A_r) 
  = \zeta_r\left(\begin{matrix}
               s_{11} &s_{12} & \cdots &s_{1r} \\
                       &s_{22} & \cdots &s_{2r} \\
                       &        & \ddots &\vdots \\
                       &       &      &s_{rr}
                 \end{matrix}\right).            
\end{align}
\end{notation}
%---  Notation: Matrices like variavles of zeta functions of root systems of type Ar  --------------------------------------------%

%---  Remark: Euler-Zagier type vs root systems Ar  -----------------------------------------------------------------------------%
\begin{rem}
Note that when $s_{ij}=0$ for $i\ge2$, 
$\zeta_r(\bfs;A_r)$ coincides with the Euler-Zagier multiple zeta function
%\footnote{We assume that $(s_{11},\dots,s_{1r}) \in \calD_r$, the region of absolute convergence of MZF.};
\begin{align*}
  \zeta_r(s_{11},\dots,s_{1r},\boldsymbol{0};A_r) = \zeta_r(s_{11},\dots,s_{1r}).
\end{align*}
And also note that the zeta-functions of root systems of type $A_r$, where the nonzero $r$ variables 
form a path starting at a diagonal entry (i.e. $s_{ii}$), moving upward or to the right, 
and ending at the top-right entry (i.e., $s_{1r}$), coincide with the Euler-Zagier multiple zeta functions.
\end{rem}
%---  Remark: Euler-Zagier type vs root systems Ar  -----------------------------------------------------------------------------%

%---  Example: zeta-functions of root systems of type A_r, figure of variables  ---%
\begin{exa}\label{Example: zeta-functions of root systems of type A_r, figure of variables}
By Proposition \ref{prop: Explicit formula for zeta function of root systems of type Ar}, 
for $s_1,s_2,s_3,s_4\in \C$ with $\re(s_i)>1$, we have  
\begin{align*}
 &\zeta_3(s_1,s_2,s_3) \\
  &=\zeta_3\left(\begin{matrix} s_1 &s_2 &s_3 \\ &0 &0\\ & &0 \end{matrix}\right) 
  =\zeta_3\left(\begin{matrix} 0 &s_2 &s_3 \\ &s_1 &0\\ & &0 \end{matrix}\right) 
  =\zeta_3\left(\begin{matrix} 0 &0 &s_3 \\ &s_1 &s_2\\ & &0 \end{matrix}\right) 
  =\zeta_3\left(\begin{matrix} 0 &0 &s_3 \\ &0 &s_2\\ & &s_1 \end{matrix}\right), \\
 &\zeta_4(s_1,s_2,s_3,s_4)
 =\zeta_4\left(\begin{matrix} 0 &0 &s_3 &s_4\\ &s_1 &s_2 &0\\ & &0 &0\\ & & &0 \end{matrix}\right)
 =\zeta_4\left(\begin{matrix} 0 &0 &s_3 &s_4\\ &0 &s_2 &0\\ & &s_1 &0\\ & & &0 \end{matrix}\right).
\end{align*}
In general, there are $2^{r-1}$ ways to express the multiple zeta function $\zeta_r(s_1,\dots,s_r)$ 
in terms of a zeta-function of root systems of type $A_r$.
It is also easy to check the following equalities:
\begin{align}\label{eqn:movements of parameters for zeta-functions of root systems}
\zeta_3\left(\begin{matrix} s_1 &0 &s_2\\ &0 &0\\ & &s_3 \end{matrix}\right)=\zeta_3\left(\begin{matrix} s_1 &0 &s_2\\ &s_3 &0\\ & &0 \end{matrix}\right),
\quad
\zeta_4\left(\begin{matrix} s_1 &s_2 &0 &s_3 \\ &0 &0 &0\\ & &0 &0\\ & & &s_4\end{matrix}\right)
  =\zeta_4\left(\begin{matrix} s_1 &s_2 &0 &s_3 \\ &0 &0 &0\\ & &s_4 &0\\ & & &0\end{matrix}\right).
\end{align}
\end{exa}
%---  Example: zeta-functions of root systems of type A_r, figure of variables  ---%

Additionally, the product of MZFs can be expressed as zeta-functions of root systems of type $A_r$;

%---  Example: product of MZFs vs root systems Ar  ---%
\begin{exa}
For $s$, $t\in\C$ with $\re(s)$, $\re(t)>1$,     
\begin{align*}
 \zeta(s)\zeta(t)
 = \sum_{m_1\in\N}\frac{1}{m_1^s}\sum_{m_2\in\N}\frac{1}{m_2^{t}} 
 = \sum_{m_1,m_2\in\N}\frac{1}{m_1^s m_2^{t} (m_1+m_2)^{0}}
 =\zeta_2\left(\begin{matrix} s &0 \\ &t \end{matrix}\right)
 =\zeta_2\left(\begin{matrix} t &0 \\ &s \end{matrix}\right).
\end{align*} 
For $s_1$, $s_2$, $s_3$ $t_1$, $t_2\in\C$ with $\re(s_i), \re(t_i)>1$, 
\begin{align*}
 &\zeta_2(s_1,s_2)\zeta(t_1)
 = \zeta_3\left(\begin{matrix} s_1 &s_2 &0 \\ &0 &0\\ & &t_1 \end{matrix}\right)
 = \zeta_3\left(\begin{matrix} t_1 &0 &0 \\ &s_1 &s_2\\ & &0 \end{matrix}\right), \\
 &\zeta_3(s_1,s_2,s_3)\zeta(t_1)
 = \zeta_4\left(\begin{matrix} s_1 &s_2 &s_3 &0 \\ &0 &0 &0\\ & &0 &0\\ & & &t_1 \end{matrix}\right) 
 = \zeta_4\left(\begin{matrix} t_1 &0 &0 &0 \\ &s_1 &s_2 &s_3\\ & &0 &0\\ & & &0 \end{matrix}\right),\\
 &\zeta_2(s_1,s_2)\zeta_2(t_1,t_2)
 = \zeta_4\left(\begin{matrix} s_1 &s_2 &0 &0 \\ &0 &0 &0\\ & &t_1 &t_2\\ & & &0 \end{matrix}\right)
 = \zeta_4\left(\begin{matrix} s_1 &s_2 &0 &0 \\ &0 &0 &0\\ & &0 &t_2\\ & & &t_1 \end{matrix}\right)
 = \zeta_4\left(\begin{matrix} t_1 &t_2 &0 &0 \\ &0 &0 &0\\ & &0 &s_2\\ & & &s_1 \end{matrix}\right).
\end{align*}
\end{exa}
%---  Example: product of MZFs vs root systems Ar  ---%

These notations are useful to understand the process of the shuffle product for MZFs. 

%---  Proposition: shuffle product relations for zeta functions with zeta-functions of type Ar and for any depths  -------------%
\begin{prop}\label{prop: shuffle product formula for mzfs with zf of type Ar}
For $p$, $q\in\N$, $(s_1,\dots,s_p)\in\C^{p}$, $(t_1,\dots,t_q)\in\C^{q}$ with 
$\re(s_p)>1$, $\re(t)>1$ and $\re(s_i)$, $\re(t_j)\ge1$ ($1\le i\le p-1$, $1\le j\le q-1$), we have
\begin{align*}%\label{eqn: shuffle product of MZFs for general case}
 &\zeta_p(s_1,\dots,s_p)\zeta_q(t_1,\dots,t_q)\\
 &= \sum_{k=0}^{\infty} 
     \left[\binom{t_q+k-1}{k} \zeta(\bfs',s_p-k;\bft';t_q+k)
           - \binom{s_p+t_q+k-1}{s_p+k} \zeta(\bfs',-k;\bft';s_p+t_q+k) \right]\notag \\
&\ + \sum_{k=0}^{\infty} 
      \left[ \binom{s_p+k-1}{k}\zeta(\bfs';\bft',t_q-k;s_p+k) 
            - \binom{s_p+t_q+k-1}{t_q+k} \zeta(\bfs';\bft',-k;s_p+t_q+k) \right],\notag
\end{align*}
where $\bfs':=(s_1,\dots,s_{p-1})$, $\bft':=(t_1,\dots,t_{q-1})$.
Here, for $\bfs=(s_1,\dots,s_p)$, $\bft=(t_1,\dots,t_q)$, the symbol $\zeta(\bfs;\bft;u)$ means the following zeta-function of root systems;
$$
\zeta(\bfs;\bft;u)
= \zeta_{p+q+1}\left(\begin{matrix}
s_{1} & \cdots & \cdots &s_{p} & 0 & \cdots & \cdots & 0 & u \\
	& 0    & \cdots & 0 & \vdots & \vdots & \vdots & \vdots & 0 \\
	& 	& \ddots &\vdots & \vdots & \vdots & \vdots & \vdots & \vdots \\
	& 	& 	 & 0 & 0 & \cdots & \cdots & 0 & 0 \\
	&	&	&	      & t_1 & \cdots & \cdots & t_q & 0 \\
	&	&	&	      &   & 0 & \cdots & 0 & 0 \\
	&	&	&	      &   & 	 & \ddots & \vdots & \vdots \\
	&	&	&	      &   & 	 & 	 & 0 & \vdots \\
	&	&	&	      &   & 	 & 	 &   & 0
\end{matrix}\right).
$$
\end{prop}
%---  Proposition: shuffle product relations for zeta functions with zeta-functions of type Ar and for any depths  -------------%

%---  Proof: shuffle product relations for zeta functions with zeta-functions of type Ar and for any depths  -------------------%
\begin{proof}
For $m_1,\dots,m_p\in\N$ and $n_1,\dots,n_q\in\N$, 
set $M_j:=\sum_{i=1}^jm_i$ ($1\le j\le p$) and $N_j:=\sum_{i=1}^jn_i$ ($1\le j\le q$).
Then summands on the left-hand side is 
\begin{align*}
 \frac{1}{M_1^{s_1}\cdots M_p^{s_p}}\cdot\frac{1}{N_1^{t_1}\cdots N_q^{t_q}}. 
\end{align*} 
Applying \eqref{eqn: IPFD, infinite partial fraction decomposition} to the product of $\frac{1}{M_p^{s_p}}$ and $\frac{1}{N_q^{t_q}}$, 
and taking summation over all $m_1,\dots,m_p\in\N$ and $n_1,\dots,n_q\in\N$ implies the claim. 
Note that we can justify the change of infinite summations 
as in Lemma \ref{lem: justification of the change of order of two infinite summations}.
\end{proof}
%---  Proof: shuffle product relations for zeta functions with zeta-functions of type Ar and for any depths  -------------------%

%---  Remark: The case s_p, t_q \in \N  -----------------------------------------------------------------------------------------%
\begin{rem}
Let $a$, $b\in\N$ and $a$, $b\ge2$. 
When $s_p=a$ $t_q=b$, then we have
\begin{align*}
 &\zeta_p(s_1,\dots,s_{p-1},a)\zeta_q(t_1,\dots,t_{q-1},b)\\
 &= \sum_{k=0}^{a-1} 
     \binom{b+k-1}{k} \zeta(\bfs',a-k;\bft';b+k)\notag 
  + \sum_{k=0}^{b-1} 
     \binom{a+k-1}{k}\zeta(\bfs';\bft',b-k;a+k).\notag
\end{align*}
Note that this relation holds for all $s_1,\dots,s_{p-1}$, $t_1,\dots,t_{q-1}\in\C$ except for the singularities.
\end{rem}
%---  Remark: The case s_p, t_q \in \N  ------------------------------------------------------------------------------------------%

In the rest of this section, we explicitly describe how to calculate the shuffle product relations inductively.

%---  Example: depth 2 times 1, 3 times 1, and 2 times 2  ------------------------------------------------------------------------%
\begin{exa} \label{example: depth 2 times 1, 3 times 1, and 2 times 2}
% \zeta_3\left(\begin{matrix} s_1 &s_2 &s_3\\ &0 &0\\ & &0 \end{matrix}\right) 
We use the notation \eqref{eqn: figure of variables of zeta functions of type Ar}.
All calculations are carried out by \eqref{eqn: IPFD, infinite partial fraction decomposition}.
\vspace{1.5mm}

%---  depth 2 times 1  ---%
\noindent
\textbf{Depth $2$ $\times$ depth $1$.} % KMT24のtheorem13.11になっているので言及しておく。
For $s_1$, $s_2$, $t\in\C$ with $\re(s_1)\ge1$, $\re(s_2)>1$ and $\re(t)>1$,  
{\small
\begin{align}\label{example: depth 2 times depth 1}
 &\zeta_2(s_1,s_2)\zeta(t) = \zeta_3\left(\begin{matrix} s_1 &s_2 &0 \\ &0 &0\\ & &t \end{matrix}\right)\\
 &= \sum_{k=0}^{\infty}
     \left[ 
      \binom{t+k-1}{k} \zeta_3\left(\begin{matrix} s_1 &s_2-k &t+k \\ &0 &0 \\ & &0 \end{matrix}\right)  
      - \binom{t+s_2+k-1}{s_2+k} \zeta_3\left(\begin{matrix} s_1 &-k &s_2+t+k \\ &0 &0 \\ & &0 \end{matrix}\right)
     \right] \notag\\
 &\quad+ \sum_{k=0}^{\infty}
     \left[ 
      \binom{s_2+k-1}{k}\zeta_3\left(\begin{matrix} s_1 &0 &s_2+k\\ &0 &0\\ & &t-k \end{matrix}\right)
      - \binom{s_2+t+k-1}{t+k}\zeta_3\left(\begin{matrix} s_1 &0 &s_2+t+k\\ &0 &0\\ & &-k \end{matrix}\right)
     \right]\notag.
\end{align}
}
Note that the summands on the first summation of the right-hand side, that is,
$\zeta_3\left(\begin{matrix} s_1 &s_2-k &t+k \\ &0 &0 \\ & &0 \end{matrix}\right)$ 
and $\zeta_3\left(\begin{matrix} s_1 &-k &s_2+t+k \\ &0 &0 \\ & &0 \end{matrix}\right)$ are the Euler-Zagier multiple zeta functions.
We calculate the summands on the second summation.
Firstly, we have
\begin{align*}%\label{eqn: depth 2 times depth 1, second step}
  \zeta_3\left(\begin{matrix} s_1 &0 &s_2+k\\ &0 &0\\ & &t-k \end{matrix}\right) 
   &= \zeta_3\left(\begin{matrix} s_1 &0 &s_2+k\\ &t-k &0\\ & &0 \end{matrix}\right) \\
  &= \sum_{j=0}^{\infty}
      \left[ 
       \binom{t-k+j-1}{j} \zeta_3\left(\begin{matrix} s_1-j &t-k+j &s_2+k\\ &0 &0\\ & &0 \end{matrix}\right)\right.\notag\\
  &\qquad\quad\left.     
       - \binom{s_1+t-k+j-1}{s_1+j} \zeta_3\left(\begin{matrix} -j &s_1+t-k+j &s_2+k\\ &0 &0\\ & &0 \end{matrix}\right)
      \right]\notag \\
  &+ \sum_{j=0}^{\infty}
      \left[ 
       \binom{s_1+j-1}{j}\zeta_3\left(\begin{matrix} 0 &s_1+j &s_2+k\\ &t-k-j &0\\ & &0 \end{matrix}\right) \right.\notag \\
  &\qquad\quad\left. 
        - \binom{s_1+t-k+j-1}{t-k+j}\zeta_3\left(\begin{matrix} 0 &s_1+t-k+j &s_2+k\\ &-j &0\\ & &0 \end{matrix}\right)
      \right]. \notag
\end{align*}
Note that the first equality holds by \eqref{eqn:movements of parameters for zeta-functions of root systems} and all members of the right-hand side are multiple zeta functions
 (see Example \ref{Example: zeta-functions of root systems of type A_r, figure of variables}). 
By using \eqref{eqn:movements of parameters for zeta-functions of root systems} again, for any $k\in\N_0$, we have 
\begin{align} \label{eqn: some variable equal non positive integer}
  \zeta_3\left(\begin{matrix} s_1 &0 &s_2+t+k\\ &0 &0\\ & &-k \end{matrix}\right)
  &=\zeta_3\left(\begin{matrix} s_1 &0 &s_2+t+k\\ &-k &0\\ & &0 \end{matrix}\right) \\
  &= \sum_{j=0}^{k}
     (-1)^j\binom{k}{j} \zeta_3\left(\begin{matrix} s_1-j &-k+j &s_2+t+k\\ &0 &0\\ & &0 \end{matrix}\right)\notag \\
  &= \sum_{j=0}^{k}
     (-1)^j\binom{k}{j} \zeta_3\left(s_1-j,-k+j, s_2+t+k\right).\notag
\end{align} 
The second equality holds because, 
for $m_1$, $m_2$, $m_3\in\N$, and $s\in\C$, we have
{\small
\begin{align*}
 &\frac{1}{m_1^{s_1}m_2^{s}(m_1+m_2+m_3)^{s_2+t}} \\
 &= \sum_{j=0}^{\infty} 
     \left[ \binom{s+j-1}{j}\mathfrak{z}_3\left( \begin{matrix} s_1-j,s+j,s_2+t \\ m_1,m_2,m_3 \end{matrix} \right) 
%\right.\\ &\qquad\quad\left.    
     - \binom{s_1+s+j-1}{s_1+j}\mathfrak{z}_3\left( \begin{matrix} -j,s_1+s+j,s_2+t \\ m_1,m_2,m_3 \end{matrix} \right) \right]\\
 &\quad+ \sum_{j=0}^{\infty} 
          \left[ \binom{s_1+j-1}{j}\mathfrak{z}_3\left( \begin{matrix} s-j,s_1+j,s_2+t \\ m_2,m_1,m_3 \end{matrix} \right)
%\right.\\ &\qquad\quad\left.         
           - \binom{s_1+s+j-1}{s+j}\mathfrak{z}_3\left( \begin{matrix} -j,s_1+s+j,s_2+t \\ m_2,m_1,m_3 \end{matrix} \right) \right].
\end{align*}
}
By the notation \eqref{eqn:general binomial coefficient},
if we put $s=-k\in\Z_{\le0}$, $\binom{s_1-k+j-1}{s_1+j}=0$ for all $j\in\N_0$, 
\begin{align*}
 \binom{-k+j-1}{j} 
 = \begin{cases}
    (-1)^j\binom{k}{j}\quad &(j=0,\dots,k) \\
    0 \quad &(j\ge k+1),
   \end{cases}
\end{align*}
and $\binom{s_1-k+j-1}{-k+j}=0$ for $0\le j\le k-1$.
So the first summation remains only
$(-1)^j\binom{k}{j}\mathfrak{z}_3\left( \begin{matrix} s_1-j,s+j,s_2+t \\ m_1,m_2,m_3 \end{matrix} \right) $
($0\le j\le k$) and the second summation vanishes. 
Hence we get the second equality of \eqref{eqn: some variable equal non positive integer}.  
Thus we obtain
\begin{align*}
&\zeta_2(s_1,s_2)\zeta(t) \\
&= \sum_{k=0}^{\infty}
     \left[ 
      \binom{t+k-1}{k} \zeta_3(s_1, s_2-k, t+k)  
      - \binom{t+s_2+k-1}{s_2+k} \zeta_3(s_1, -k, s_2+t+k)
     \right] \\
%2nd-summation
 &\quad+ \sum_{k=0}^{\infty}
     \left[ 
      \binom{s_2+k-1}{k} \sum_{j=0}^{\infty}
      \left[ 
       \binom{t-k+j-1}{j} \zeta_3(s_1-j, t-k+j, s_2+k)\right. \right. \\
  &\hspace{4.5cm}\left.- \binom{s_1+t-k+j-1}{s_1+j} \zeta_3(-j, s_1+t-k+j, s_2+k)
      \right] \\
  &\quad\hphantom{\quad+ \sum_{k=0}^{\infty}\left[\right.}+ \binom{s_2+k-1}{k} \sum_{j=0}^{\infty}
      \left[ 
       \binom{s_1+j-1}{j}\zeta_3(t-k-j, s_1+j, s_2+k) \right. \\
  &\hspace{4.5cm}\left. 
        - \binom{s_1+t-k+j-1}{t-k+j}\zeta_3(-j, s_1+t-k+j, s_2+k)
      \right] \\
%2nd-summation, 2nd-term
 &\hspace{2cm} \left.
      - \binom{s_2+t+k-1}{t+k} \sum_{j=0}^{k}
     (-1)^j\binom{k}{j} \zeta_3\left(s_1-j,-k+j, s_2+t+k\right)
     \right].
\end{align*}
When we put $(s_1,s_2)=(b,c)$ and $t=a$ ($b\in\N$, $c,a\in\Z_{\ge2}$), 
the shuffle product relations for MZFs recover the original shuffle product relations for MZVs;
\begin{align*}
\zeta(b,c)\zeta(a)
=&\sum_{k=0}^{c-1}\binom{a+k-1}{k}\zeta(b,c-k,a+k)\\
&+\sum_{k=0}^{a-1}\sum_{j=0}^{b-1}\binom{c+k-1}{k}\binom{a-k+j-1}{j}\zeta(b-j,a-k+j,c+k)\\
&+\sum_{k=0}^{a-1}\sum_{j=0}^{a-k-1}\binom{c+k-1}{k}\binom{b+j-1}{j}\zeta(a-k-j,b+j,c+k).
\end{align*}
%---  depth 2 times 1  ---%

\vspace{1.5mm}\noindent
%---  depth 3 times 1  ---%
\textbf{Depth $3$ $\times$ depth $1$.}
For $s_1$, $s_2$, $s_3$, $t\in\C$ ($\re(s_1)$, $\re(s_2)\ge1$, $\re(s_3)$, $\re(t)>1$),  
\begin{align}\label{example: depth 3 times depth 1}  
  &\zeta_3(s_1,s_2,s_3)\zeta(t) = \zeta_4\left(\begin{matrix} s_1 &s_2 &s_3 &0 \\ &0 &0 &0\\ & &0 &0\\ & & &t \end{matrix}\right) \\
  &= \sum_{k=0}^{\infty}
      \left[ 
       \binom{t+k-1}{k} \zeta_4\left(\begin{matrix} s_1 &s_2 &s_3-k &t+k \\ &0 &0 &0\\ & &0 &0\\ & & &0\end{matrix}\right) \right.\notag\\ &\hspace{3cm}\left.     
       - \binom{s_3+t+k-1}{s_3+k} \zeta_4\left(\begin{matrix} s_1 &s_2 &-k &s_3+t+k \\ &0 &0 &0\\ & &0 &0\\ & & &0\end{matrix}\right)
      \right] \notag\\
  &\quad+ \sum_{k=0}^{\infty}
      \left[ 
       \binom{s_3+k-1}{k} \zeta_4\left(\begin{matrix} s_1 &s_2 &0 &s_3+k \\ &0 &0 &0\\ & &0 &0\\ & & &t-k\end{matrix}\right) \right.\notag\\ &\hspace{3cm}\left.
       - \binom{s_3+t+k-1}{t+k} \zeta_4\left(\begin{matrix} s_1 &s_2 &0 &s_3+t+k \\ &0 &0 &0\\ & &0 &0\\ & & &-k\end{matrix}\right)
      \right].\notag 
\end{align}
Note that the summands on the first summation of the right-hand side
are multiple zeta functions.
By \eqref{eqn:movements of parameters for zeta-functions of root systems}, for any $k\in\N_0$, we have 
\begin{align}\label{eqn:calculation of zeta function of root systems with (2*1,1)}
 &\zeta_4\left(\begin{matrix} s_1 &s_2 &0 &s_3+k \\ &0 &0 &0\\ & &0 &0\\ & & &t-k\end{matrix}\right)
  =\zeta_4\left(\begin{matrix} s_1 &s_2 &0 &s_3+k \\ &0 &0 &0\\ & &t-k &0\\ & & &0\end{matrix}\right)
\intertext{By using \eqref{eqn: IPFD, infinite partial fraction decomposition}, we get} 
 &= \sum_{j=0}^{\infty}
 \left[ 
  \binom{t-k+j-1}{j} \zeta_4\left(\begin{matrix} s_1 &s_2-j &t-k+j &s_3+k \\ &0 &0 &0\\ & &0 &0\\ & & &0\end{matrix}\right) \right. \notag \\
 &\qquad\quad\left. 
  - \binom{s_2+t-k+j-1}{s_2+j} \zeta_4\left(\begin{matrix} s_1 &-j &s_2+t-k+j &s_3+k \\ &0 &0 &0\\ & &0 &0\\ & & &0\end{matrix}\right)
 \right] \notag\\
 &+ \sum_{j=0}^{\infty}
 \left[ 
  \binom{s_2+j-1}{j} \zeta_4\left(\begin{matrix} s_1 &0 &s_2+j &s_3+k \\ &0 &0 &0\\ & &t-k-j &0\\ & & &0\end{matrix}\right) \right. \notag \\
 &\qquad\quad 
 \left.- \binom{s_2+t-k+j-1}{t-k+j} \zeta_4\left(\begin{matrix} s_1 &0 &s_2+t-k+j &s_3+k \\ &0 &0 &0\\ & &-j &0\\ & & &0\end{matrix}\right)
 \right], \notag 
\end{align}
and similarly to \eqref{eqn: some variable equal non positive integer} in the case of depth $2$ $\times$ depth $1$, we obtain
\begin{align*}
 \zeta_4\left(\begin{matrix} s_1 &s_2 &0 &s_3+t+k \\ &0 &0 &0\\ & &0 &0\\ & & &-k\end{matrix}\right)
  &=\zeta_4\left(\begin{matrix} s_1 &s_2 &0 &s_3+t+k \\ &0 &0 &0\\ & &-k &0\\ & & &0\end{matrix}\right) \\
 &= \sum_{j=0}^k (-1)^j\binom{k}{j} \zeta_4\left(s_1, s_2-j,-k+j, s_3+t+k\right).
\end{align*}
Furthermore, by \eqref{eqn:movements of parameters for zeta-functions of root systems}, for any $k$, $j\in\N_0$, 
\begin{align}\label{eqn:calculation of zeta function of root systems with (1*1,2)}
 &\zeta_4\left(\begin{matrix} s_1 &0 &s_2+j &s_3+k \\ &0 &0 &0\\ & &t-k-j &0\\ & & &0\end{matrix}\right)
  =\zeta_4\left(\begin{matrix} s_1 &0 &s_2+j &s_3+k \\ &t-k-j &0 &0\\ & &0 &0\\ & & &0\end{matrix}\right).
\intertext{By using \eqref{eqn: IPFD, infinite partial fraction decomposition}, we get}
 &= \sum_{l=0}^{\infty}
     \left[
      \binom{t-k-j+l-1}{l} \zeta_4\left(\begin{matrix} s_1-l &t-k-j+l &s_2+j &s_3+k \\ &0 &0 &0\\ & &0 &0\\ & & &0\end{matrix}\right) \right. \notag\\
 &\qquad\quad\left. 
      - \binom{s_1+t-k-j+l-1}{s_1+l} \zeta_4\left(\begin{matrix} -l &s_1+t-k-j+l &s_2+j &s_3+k \\ &0 &0 &0\\ & &0 &0\\ & & &0\end{matrix}\right)
     \right] \notag\\
 &\quad+\sum_{l=0}^{\infty}
         \left[
          \binom{s_1+l-1}{l}\zeta_4\left(\begin{matrix} 0 &s_1+l &s_2+j &s_3+k \\ &t-k-j-l &0 &0\\ & &0 &0\\ & & &0\end{matrix}\right) \right.\notag\\
 &\qquad\quad\left.         
          - \binom{s_1+t-k-j+l-1}{t-k-j+l}\zeta_4\left(\begin{matrix} 0 &s_1+t-k-j+l &s_2+j &s_3+k \\ &-l &0 &0\\ & &0 &0\\ & & &0\end{matrix}\right)
         \right]. \notag
\end{align}
Note that %$\zeta_4\left(\begin{matrix} 0 &s_1+l &s_2+j &s_3+k \\ &t-k-j-l &0 &0\\ & &0 &0\\ & & &0\end{matrix}\right)$ and $\zeta_4\left(\begin{matrix} 0 &s_1+t-k-j+l &s_2+j &s_3+k \\ &-l &0 &0\\ & &0 &0\\ & & &0\end{matrix}\right)$
the summands on the second summation of the right-hand side 
are multiple zeta functions, see Example \ref{Example: zeta-functions of root systems of type A_r, figure of variables}.
And similarly to \eqref{eqn: some variable equal non positive integer} in the case of depth $2$ $\times$ depth $1$, we obtain
{\small
\begin{align*}
 \zeta_4\left(\begin{matrix} s_1 &0 &s_2+t-k+j &s_3+k \\ &0 &0 &0\\ & &-j &0\\ & & &0\end{matrix}\right)
 &=\zeta_4\left(\begin{matrix} s_1 &0 &s_2+t-k+j &s_3+k \\ &-j &0 &0\\ & &0 &0\\ & & &0\end{matrix}\right) \\
 &= \sum_{l=0}^j (-1)^{l}
     \binom{j}{l}\zeta_4\left(s_1-l, -j+l, s_2+t-k+j, s_3+k\right).
\end{align*}
}
Thus one sees that the product $\zeta_3(s_1,s_2,s_3)\zeta(t)$ can be expressed as infinite summations of MZFs.
%---  depth 3 times 1  ---%

\vspace{1.5mm}\noindent
%---  depth 2 times 2  ---%
\textbf{Depth $2$ $\times$ depth $2$.} For $s_1,s_2,t_1,t_2\in\C$ ($\re(s_1)$, $\re(t_1)\ge1$, $\re(s_2)$, $\re(t_2)>1$),
\begin{align*}
\zeta_2(s_1,s_2)\zeta_2(t_1,t_2) 
=\zeta_4\left(\begin{matrix} s_1 &s_2 &0 &0\\ &0 &0 &0\\ & &t_1 &t_2\\ & & &0\end{matrix}\right)
\end{align*}
can be expressed by the following four zeta-functions of type $A_r$;
{\small
\begin{align*}
\zeta_4\left(\begin{matrix} s_1 &s_2-k &0 &t_2+k\\ &0 &0 &0\\ & &t_1 &0\\ & & &0\end{matrix}\right),\quad
\zeta_4\left(\begin{matrix} s_1 &-k &0 &s_2+t_2+k\\ &0 &0 &0\\ & &t_1 &0\\ & & &0\end{matrix}\right),& \\
\zeta_4\left(\begin{matrix} s_1 &0 &0 &s_2+k\\ &0 &0 &0\\ & &t_1 &t_2-k\\ & & &0\end{matrix}\right),\quad
\zeta_4\left(\begin{matrix} s_1 &0 &0 &s_2+t_2+k\\ &0 &0 &0\\ & &t_1 &-k\\ & & &0\end{matrix}\right).&
\end{align*}
}
We see that the first two can be expressed in the same way as \eqref{eqn:calculation of zeta function of root systems with (2*1,1)}, by infinite summations of MZFs.
Since the third and fourth members can be computed in the same way, we only explain the third one.
Similarly to \eqref{eqn:movements of parameters for zeta-functions of root systems}, we have
$$
\zeta_4\left(\begin{matrix} s_1 &0 &0 &s_2+k\\ &0 &0 &0\\ & &t_1 &t_2-k\\ & & &0\end{matrix}\right)
=\zeta_4\left(\begin{matrix} s_1 &0 &0 &s_2+k\\ &t_1 &t_2-k&0\\ & &0 & 0\\ & & &0\end{matrix}\right),
$$
so by \eqref{eqn: IPFD, infinite partial fraction decomposition}, 
this function can be expressed by the following four zeta-functions of type $A_r$;
{\small
\begin{align*}
\zeta_4\left(\begin{matrix} s_1-j &0 &t_2-k+j &s_2+k\\ &t_1 &0 &0\\ & &0 &0\\ & & &0\end{matrix}\right),\,
\zeta_4\left(\begin{matrix} -j &0 &s_1+t_2-k+j &s_2+k\\ &t_1 &0 &0\\ & &0 &0\\ & & &0\end{matrix}\right), &\\
\zeta_4\left(\begin{matrix} 0 &0 &s_1+j &s_2+k\\ &t_1 &t_2-k-j &0\\ & &0 &0\\ & & &0\end{matrix}\right),\,
\zeta_4\left(\begin{matrix} 0 &0 &s_1+t_2-k+j &s_2+k\\ &t_1 &-j &0\\ & &0 &0\\ & & &0\end{matrix}\right).&
\end{align*}
}
We see that the first two can be expressed in the same way as \eqref{eqn:calculation of zeta function of root systems with (1*1,2)}, 
by infinite summations of MZFs.
The last two are nothing but MZFs, see Example \ref{Example: zeta-functions of root systems of type A_r, figure of variables}.
Hence, the product $\zeta_2(s_1,s_2)\zeta_2(t_1,t_2)$ can be expressed by an infinite sum of MZFs.
%---  depth 2 times 2  ---%
\end{exa}
%---  Example: epth 2 times 1, 3 times 1, and 2 times 2  ------------------------------------------------------------------------%

%---  Remark: Bradely and Zhou's work  ------------------------------------------------------------------------------------------%
\begin{rem}
In \cite{BZ10}, Bradley and Zhou proved any Mordell-Tornheim multiple zeta values can be expressed as linear combinations of MZVs.
More generally, by Proposition \ref{prop: shuffle product formula for mzfs with zf of type Ar} 
and the inductive calculation, one can show that Mordell-Tornheim MZF can be expressed as 
infinite summations in terms of Euler-Zagier multiple zeta functions.
\end{rem}
%---  Remark: Bradely and Zhou's work  ------------------------------------------------------------------------------------------%
%---  subsection: zeta functions of root systems  -------------------------------------------------------------------------------%
%===  SECTION END  ==============================================================================================================%

%===  Section: Application and problems  ========================================================================================%
\section{Applications}\label{sec: Applications and problems}
In this section, we provide some applications.
We first mention the double shuffle relations for MZFs as functional relations.
In subsection \ref{subsec:Functional relations for the zeta functions of root systems}, 
we demonstrate that the results by Komori, Matsumoto and Tsumura (\cite{KMT11}) can be deduced from 
Proposition \ref{prop: shuffle product formula for mzfs with zf of type Ar}. 
In subsection \ref{subsec:Sum formula and extended double shuffle relation for MZFs}, 
we show the sum formulas for MZFs studied by Hirose, Murahara and Onozuka (\cite{HMO18}) in the double case. 
We also prove that the extended double shuffle relations for MZFs.
%In subsection \ref{subsec:Recurrence relations for MZFs},
%we prove the equivalence between the recurrence relations for MZFs and shuffle product relations for MZFs in the special case.
%---  Shuffle product for multiple zeta functions  ------------------------------------------------------------------------------%

We give the direct consequence of Theorem \ref{Main theorem}.
Since we have already known that the stuffle product relations for MZFs hold, therefore;

%---  Theorem: functional double shuffle product relation  -----------------------------------------------------------------------%
\begin{thm}\label{thm: functional double shuffle product relation}
Functional double shuffle relations hold for MZFs. 
Namely, for $s_1$,$\dots$,$s_p$, $t_1$,$\dots$,$t_q\in\C$ with 
\begin{align*}
 \re(s_p),\,\re(t_q)>1,\  \re(s_i),\,\re(t_j)\ge1 \quad (1\le i\le p-1,\  1\le j \le q-1),
\end{align*}
the product of two MZFs $\zeta_p(s_1,\dots,s_p)\zeta_q(t_1,\dots,t_q)$ 
can be calculated in two ways, stuffle product relations and shuffle product relations
and those two relations implies a non-trivial functional relation for MZFs. 
\end{thm}
%---  Theorem: functional double shuffle product relation  -----------------------------------------------------------------------%

Note that if one specializes the shuffle product relations for MZFs to positive integer indices, 
we get original shuffle product relations for MZVs. 
Thus, functional double shuffle relation for MZFs implies the finite double shuffle relations for MZVs.
The simplest example is Corollary \ref{cor: functional double shuffle relations for double}.
However, we omit examples for higher depths 
because we need much space to write down them and they are so complicated even in the case depth $2$ $\times$ depth $1$, 
the case depth $3$ $\times$ depth $1$ or the case depth $2$ $\times$ depth $2$.
Those cases can be calculated using Example \ref{example: depth 2 times 1, 3 times 1, and 2 times 2} with stuffle product relations.

%---  Remark: desingularization  -------------------------------------------------------------------------------------------------%
\begin{rem}
In (\cite{KS23}), we proved that the values of desingularized MZFs at any integer points satisfy the shuffle product relation.
So it seems interesting if our previous results can be extended to the functional relations for desingularized MZFs.
\end{rem}
%---  Remark: desingularization  ------------------------------------------------------------------------------------------------%

%---  Shuffle product for multiple zeta functions  -------------------------------------------------------------------------------%
%===  Section: Application and problems  =========================================================================================%

%===  Section: Application and problems  =========================================================================================%
%---  subsection: Functional relations for the zeta functions of root systems  ---------------------------------------------------%
\subsection{Functional relations for the zeta-functions of root systems}
\label{subsec:Functional relations for the zeta functions of root systems}
As mentioned before, we show that Proposition \ref{prop: shuffle product formula for mzfs with zf of type Ar} yields the works of Komori, Matsumoto, and Tsumura.
By putting $s_2,t\in\N_{\ge2}$ in \eqref{example: depth 2 times depth 1} and 
putting $s_3$, $t\in\N_{\ge2}$ in \eqref{example: depth 3 times depth 1}, 
we get the following.
%---  Corollary: generalization of KMT's results  --------------------------------------------------------------------------------%
\begin{cor}[{cf. \cite[(42) and (43)]{KMT11}}]\label{KMT's results}
For $b$, $c\in\N_{\ge2}$ and $s\in\C$ with $\re(s)>1$, we have
\begin{align*}
% &\zeta_3(s,b,c)+\zeta_3(s,c,b)+\zeta_3(c,s,b)+\zeta_2(s,b+c)+\zeta_2(s+c,b) \\
\zeta_2(s,b)\zeta(c)
 = \sum_{k=0}^{b-1}\binom{c+k-1}{k}\zeta_3(s,b-k,c+k)
  + \sum_{k=0}^{c-1}\binom{b+k-1}{k}\zeta_3\left(\begin{matrix} s &0 &b+k\\ &0 &0\\ & &c-k \end{matrix}\right) \notag.
\end{align*} 
For $c$, $d\in\N_{\ge2}$ and $s_1$, $s_2\in\C$ with $\re(s_1),\re(s_2)>1$, we have
\begin{align*}
 %&\zeta_4(s_1,s_2,c,d)+\zeta_4(s_1,s_2,d,c)+\zeta_4(s_1,d,s_2,c)+\zeta_4(d,s_1,s_2,c) \\
 %&\quad +\zeta_3(s_1,s_2,c+d)+\zeta_3(s_1,s_2+d,c)+\zeta_3(s_1+d,s_2,c) \notag\\
 &\zeta_3(s_1,s_2,c)\zeta(d) \\
 &= \sum_{k=0}^{c-1}\binom{d+k-1}{k}\zeta_4(s_1,s_2,c-k,d+k)
  +  \sum_{k=0}^{d-1}\binom{c+k-1}{k}
      \zeta_4\left(\begin{matrix} s_1 &s_2 &0 &c+k\\ &0 &0 &0\\ & &0 &0 \\ & & &d-k \end{matrix}\right). \notag
\end{align*}
\end{cor}
%---  Corollary: generalization of KMT's results  --------------------------------------------------------------------------------%

By applying stuffle product relations to the left-hand side of the above equations,
we get Theorem 3 in \cite{KMT11}.

% In \cite{KMT24}, some functional relations for zeta functions of root systems are described by the action of Weyl groups.
% So we leave the somewhat general question here; 
% \begin{prob}
%  What is the connection between the shuffle product relations for MZfs and the theory of zeta-functions of root systems? 
% \end{prob}
% from the view point of root systems, 
%---  subsection: Functional relations for the zeta functions of root systems  ---------------------------------------------------%
%===  Section: Application and problems  ========================================================================================%

%===  Section: Application and problems  =========================================================================================%
%---  subsection: Sum formula for MZFs  ------------------------------------------------------------------------------------------%
\subsection{Sum formula and extended double shuffle relation for MZFs}
\label{subsec:Sum formula and extended double shuffle relation for MZFs}
We next study the connection between functional double shuffle relations and sum formulas for double zeta functions.
The sum formulas for multiple zeta functions was proved by Hirose, Murahara and Onozuka by direct calculation.
The example of the simplest case is;
%---  Theorem: Sum formula for MZFs  ---------------------------------------------------------------------------------------------%
\begin{cor}[{\cite[Theorem 1.2]{HMO18}}]\label{thm: functional sum formula}
For $s\in\C$ with $\re(s)>1$ and $s\ne2$,
\begin{align}\label{eqn: functional sum formula double case}
  \zeta(s)=\sum_{k=0}^{\infty}\left[ \zeta_2(s-k-2, k+2) - \zeta_2(-k, s+k) \right].
\end{align}  
\end{cor}
%---  Theorem: Sum formula for MZFs  --------------------------------------------------------------------------------------------%

We here prove this claim as a consequence of functional double shuffle relations for double zeta functions.

%---  Proof of sum formula for MZFs ---------------------------------------------------------------------------------------------%
{\it Proof. }% of Theorem \ref{thm: functional sum formula}
By (IPFD) with $t=1$, we have 
\begin{align*}
 \frac{1}{x^sy} 
 &= \sum_{k=0}^{\infty} 
     \left[ \frac{1}{x^{s-k}(x+y)^{1+k}} - \frac{1}{x^{-k}(x+y)^{s+1+k}} \right]\\
 &\quad+ \sum_{k=0}^{\infty} 
          \left[ \binom{s+k-1}{k}\frac{1}{y^{1-k}(x+y)^{s+k}} - \binom{s+k}{1+k}\frac{1}{y^{-k}(x+y)^{s+1+k}} \right] \\
 &= \frac{1}{x^{s}(x+y)} + \frac{1}{y(x+y)^{s}} + 
     \sum_{k=0}^{\infty} 
     \left[ \frac{1}{x^{s-k-1}(x+y)^{2+k}} - \frac{1}{x^{-k}(x+y)^{s+1+k}} \right].
\end{align*}
Note that \eqref{eqn: IPFD, infinite partial fraction decomposition} holds for $\re(s)$, $\re(t)>0$. 
By taking summation with respect to $1\le x,\, y\le N$ ($N\in\N$), we have 
\begin{align}\label{eqn: sum formula for double case}
\begin{split}
 &\sum_{1\le x,y\le N} \left[\frac{1}{x^sy} - \frac{1}{x^{s}(x+y)} - \frac{1}{y(x+y)^{s}} \right] \\
 &= \sum_{k=0}^{\infty} \sum_{1\le x,y\le N}
     \left[ \frac{1}{x^{s-k-1}(x+y)^{2+k}} - \frac{1}{x^{-k}(x+y)^{s+1+k}} \right]. 
\end{split}
\end{align}
Then, it is easy to check that the left-hand side of \eqref{eqn: sum formula for double case} coincides with the following
\begin{align*}
 \sum_{x=1}^N\frac{1}{x^{s+1}}.
\end{align*}
Thus, the left-hand side of \eqref{eqn: sum formula for double case} will be $\zeta(s+1)$ as $N\rightarrow\infty$.
Since we can show that the right-hand side of \eqref{eqn: sum formula for double case} converges absolutely
as in Lemma \ref{lem: justification of the change of order of two infinite summations}, that will be 
\begin{align*}
 \sum_{k=0}^{\infty} \left[\zeta_2(s-k-1,k+2) - \zeta_2(-k,s+k+1) \right].
\end{align*}
Hence, we obtain the sum formula \eqref{eqn: functional sum formula double case} by putting $s\mapsto s-1$.\qed
\begin{comment}
By Corollary \ref{cor: functional double shuffle relations for double}, we have 
\begin{align}\label{eqn: functional double relation for double zeta in the section application}
 &\zeta_2(s,t) + \zeta_2(t,s) + \zeta(s+t) \\
 &= \sum_{k=0}^{\infty} \left[\binom{t+k-1}{k}\zeta_2(s-k,t+k) - \binom{s+t+k-1}{s+k}\zeta_2(-k,s+t+k) \right]\notag \\
 &\quad+ \sum_{k=0}^{\infty} \left[ \binom{s+k-1}{k}\zeta_2(t-k,s+k) - \binom{s+t+k-1}{t+k}\zeta_2(-k,s+t+k) \right]\notag.
\end{align}
for $s$, $t\in\C$ with $\re(s)$, $\re(t)>1$.
On the right-hand side, we find two terms $\zeta_2(s,t)$ and $\zeta_2(t,s)$.
Those terms also appear in the left-hand side and can be cancelled.
So the term $\zeta(s+t)$ remains only on the left-hand side.
Moreover, since $\zeta_2(s-k,t+k)$ are regular at $t=1$ when $k\ge 1$, the first summation of the right-hand side will be
\begin{align*}
 \sum_{k=0}^{\infty} \left[\zeta_2(s-k-1,k+2) - \zeta_2(-k,s+k+1) \right].
\end{align*} 
Since we already eliminate the term $\zeta_2(t,s)$, 
telescoping happens on the second summation of the right-hand side 
of \eqref{eqn: functional double relation for double zeta in the section application} when $t=1$
and therefore the second summation vanishes.
\end{comment}
%---  Proof of sum formula for MZFs ---------------------------------------------------------------------------------------------%

It is known that sum formulas follow from the extended double shuffle relations(EDSR) for MZVs (\cite{IKZ06}).
We next focus on the EDSR for MZFs.

For any index $\mathbf{k}=(k_1,\dots,k_{n-1},k_n)\in\N^n$, consider the following MZF with single complex variable $s$:
$$
\zeta_n(k_1,\dots,k_{n-1},k_n+s)
=\sum_{0<m_1<\cdots<m_n}\frac{1}{m_1^{k_1}\cdots m_{n-1}^{k_{n-1}}m_{n}^{k_n+s}}.
$$
Put $Z_{\mathbf{k}}^\shuffle(T)$ to be the polynomial introduced in \cite[\S 2]{IKZ06}.
Note that we have $Z_{\mathbf{k}}^\shuffle(T)=\zeta(\mathbf{k})$ for $k_n\ge2$,
and note that the constant term $Z_{\mathbf{k}}^\shuffle(0)=Z_{\mathbf{k}}^\shuffle(T)|_{T=0}$ of the polynomial $Z_{\mathbf{k}}^\shuffle(T)$ is sometimes called the {\it shuffle regularized MZV}.
\begin{thm}[{cf. \cite[Theorem 4.1.(ii)]{AK-MLVs}}]%\footnote{This claim is accurately stated in \cite[Theorem 1.4.27.(i)]{AK05}, but the book \cite{AK05} is written in Japanese.}
\label{thm:Laurent series expansion of 1variableMZF by AK}
Write the polynomial $Z_{\mathbf{k}}^\shuffle(T)$ as
$$
Z_{\mathbf{k}}^\shuffle(T)
=\sum_{j=0}^\nu c_j(\mathbf{k})\frac{T^j}{j!}.
$$
%where $\gamma$ is Euler constant.
Then the principal part of $\Gamma(s+1)\zeta_n(k_1,\dots,k_{n-1},k_n+s)$ at $s=0$ is given by
$$
\Gamma(s+1)\zeta_n(k_1,\dots,k_{n-1},k_n+s)
=\sum_{j=0}^\nu \frac{c_j(\mathbf{k})}{s^j} + O(s) \qquad (s\rightarrow 0).
$$
\end{thm}

Note that the constant term $c_0(\mathbf{k})$ means the shuffle regularized MZV $Z_{\mathbf{k}}^\shuffle(0)$.
By using the above theorem, we obtain the following corollary.
\begin{cor}\label{cor: EDSR for MZFs}
Functional double shuffle relations for MZFs induce the extended double shuffle relations for MZVs.
%, that is,
%Theorem \ref{thm: functional double shuffle product relation} induces the extended double shuffle relations for MZVs.
\end{cor}
\begin{proof}
%By Proposition \ref{prop: IPFD, infinite partial fraction decomposition} and \eqref{eqn: classical partial fractin decompositon}, we know that (finite) double shuffle relations for MZVs are induced from functional double shuffle relations for MZFs.
For $k_1,\dots,k_{p-1},l_1,\dots,l_{q-1}\in\N,k_p\in\N_{\ge2}$, we compute the constant term of the Laurent series expansion of
\begin{equation}\label{eqn:Gamma zeta_p zeta_q}
\Gamma(s+1)\zeta_p(k_1,\dots,k_{p-1},k_p)\zeta_q(l_1,\dots,l_{q-1},1+s)
\end{equation}
at $s=0$ in two ways by using Theorem \ref{thm:Laurent series expansion of 1variableMZF by AK}.
%For our simplicity, we denote
%$$
%\zeta\left( \sum_{\bfs}\bfs \right):=\sum_{\bfs}\zeta\left( \bfs \right),
%\qquad c_j\left( \sum_{\mathbf{k}}\mathbf{k} \right):=\sum_{\mathbf{k}}c_j(\mathbf{k}).
%$$

Firstly, by stuffle product relations for MZFs, we calculate
\begin{align*}
&\zeta_p(k_1,\dots,k_{p-1},k_p)\cdot\zeta_q(l_1,\dots,l_{q-1},1+s) \\
&=\zeta((k_1,\dots,k_{p-1},k_p)*(l_1,\dots,l_{q-1},1+s)) \\
&=\zeta((k_1,\dots,k_{p-1})*(l_1,\dots,l_{q-1},1+s),k_p) \\
&\quad +\zeta((k_1,\dots,k_{p-1},k_p)*(l_1,\dots,l_{q-1}),1+s) \\
&\qquad +\zeta((k_1,\dots,k_{p-1})*(l_1,\dots,l_{q-1}),k_p+1+s).
\end{align*}
On the second term of the right hand side, by Theorem \ref{thm:Laurent series expansion of 1variableMZF by AK} we have
\begin{align*}
&\Gamma(s+1)\zeta((k_1,\dots,k_{p-1},k_p)*(l_1,\dots,l_{q-1}),1+s) \\
&=(\mbox{pole parts}) + c_0((k_1,\dots,k_{p-1},k_p)*(l_1,\dots,l_{q-1}),1) + O(s).
\end{align*}
On the first and third terms of the right hand side, when $s$ goes to $0$, the limit values of 
{\small
$$
\Gamma(s+1)\zeta((k_1,\dots,k_{p-1})*(l_1,\dots,l_{q-1},1+s),k_p),
\quad \Gamma(s+1)\zeta((k_1,\dots,k_{p-1})*(l_1,\dots,l_{q-1}),k_p+1+s)
$$
}
are expressed respectively as linear combinations of MZVs:
\begin{align*}
\zeta((k_1,\dots,k_{p-1})*(l_1,\dots,l_{q-1},1),k_p),
\quad \zeta((k_1,\dots,k_{p-1})*(l_1,\dots,l_{q-1}),k_p+1).
\end{align*}
So the constant term of the Laurent series expansion of \eqref{eqn:Gamma zeta_p zeta_q} at $s=0$ is obtained as follows:\footnote{Note that $\zeta(k_1,\dots,k_{p-1},k_p)=c_0(k_1,\dots,k_{p-1},k_p)$ for $k_p\ge2$.}
\begin{align}\label{eqn:stuffle-reg product}
&\zeta((k_1,\dots,k_{p-1})*(l_1,\dots,l_{q-1},1),k_p) +c_0((k_1,\dots,k_{p-1},k_p)*(l_1,\dots,l_{q-1}),1) \\
&\quad +\zeta((k_1,\dots,k_{p-1})*(l_1,\dots,l_{q-1}),k_p+1) \notag\\
&=c_0((k_1,\dots,k_{p-1},k_p)*(l_1,\dots,l_{q-1},1)). \notag
\end{align}

Secondly, by Proposition \ref{prop: shuffle product formula for mzfs with zf of type Ar}, we calculate
\begin{align*}
&\zeta_p(k_1,\dots,k_{p-1},k_p)\cdot\zeta_q(l_1,\dots,l_{q-1},1+s) \\
&= \sum_{\tau=0}^{k_p-1} 
	\binom{1+s+\tau-1}{\tau} \zeta(\mathbf{k}',k_p-\tau;\mathbf{l}';1+s+\tau) \\
&\quad + \sum_{\tau=0}^{\infty} 
	\left[ \binom{k_p+\tau-1}{\tau}\zeta(\mathbf{k}';\mathbf{l}',1+s-\tau;k_p+\tau) \right. \\
&\hspace{2cm} \left. - \binom{k_p+1+s+\tau-1}{1+s+\tau} \zeta(\mathbf{k}';\mathbf{l}',-\tau;k_p+1+s+\tau) \right],
\end{align*}
where $\mathbf{k}'=(k_1,\dots,k_{p-1})$, $\mathbf{l}'=(l_1,\dots,l_{q-1})$.
On the first summation of the right hand side, by the classical PFD \eqref{eqn: classical partial fractin decompositon} and Theorem \ref{thm:Laurent series expansion of 1variableMZF by AK}, we have
\begin{align*}
&\Gamma(s+1)\zeta(\mathbf{k}',k_p-\tau;\mathbf{l}';1+s+\tau) \\
&=\Gamma(s+1)\zeta((k_1,\dots,k_{p-1},k_p-\tau)\shuffle(l_1,\dots,l_{q-1}),1+s+\tau) \\
&=(\mbox{pole parts}) + c_0((k_1,\dots,k_{p-1},k_p-\tau)\shuffle(l_1,\dots,l_{q-1}),1+\tau) + O(s)
\end{align*}
for $0\le \tau\le k_p-1$.
On the second summation of the right hand side, when $s$ goes to $0$, the limit value of this summation becomes the single term $\zeta(\mathbf{k}';\mathbf{l}',1;k_p)$ and is calculated by the classical PFD \eqref{eqn: classical partial fractin decompositon} as follows:
\begin{align*}
\zeta(\mathbf{k}';\mathbf{l}',1;k_p)
=\zeta((k_1,\dots,k_{p-1})\shuffle(l_1,\dots,l_{q-1},1),k_p).
\end{align*}
So the constant term of the Laurent series expansion of \eqref{eqn:Gamma zeta_p zeta_q} at $s=0$ is obtained as follows:
\begin{align}\label{eqn:shuffle-reg product}
&\sum_{\tau=0}^{k_p-1}c_0((k_1,\dots,k_{p-1},k_p-\tau)\shuffle(l_1,\dots,l_{q-1}),1+\tau) \\
&\quad+\zeta((k_1,\dots,k_{p-1})\shuffle(l_1,\dots,l_{q-1},1),k_p) \notag \\
&=c_0\left( \sum_{\tau=0}^{k_p-1}((k_1,\dots,k_{p-1},k_p-\tau)\shuffle(l_1,\dots,l_{q-1}),1+\tau) \right. \notag \\
&\hspace{3cm} \left. \vphantom{\sum_{\tau=0}^{k_p-1}}+((k_1,\dots,k_{p-1})\shuffle(l_1,\dots,l_{q-1},1),k_p) \right) \notag \\
&=c_0((k_1,\dots,k_{p-1},k_p)\shuffle(l_1,\dots,l_{q-1},1)). \notag
\end{align}
In the last equality, we used \cite[Proposition 1]{KMT11}. %\footnote{Note that the order of the index of MZV in this paper is opposite to the one in \cite{KMT11}.}
Hence, by comparing \eqref{eqn:stuffle-reg product} and \eqref{eqn:shuffle-reg product}, we obtain
$$
c_0((k_1,\dots,k_{p-1},k_p)*(l_1,\dots,l_{q-1},1) - (k_1,\dots,k_{p-1},k_p)\shuffle(l_1,\dots,l_{q-1},1))=0,
$$
that is, the extended double shuffle relations for MZVs are induced from functional double shuffle relations for MZFs.
\end{proof}

%---  subsection: Sum formula for MZFs  -----------------------------------------------------------------------------------------%
%===  Section: Application and problems  ========================================================================================%

%===  Section: Application and problems  ========================================================================================%
\section{Problems}\label{sec:Concluding remarks}

Akiyama, Egami, and Tanigawa (\cite{AET01}) used the so-called Euler-Maclaurin summation formula and
essentially proved the recurrence relations \eqref{eqn:recurrence relations with non-positive integer sr} for MZFs. 
Almost the same time, Matsumoto (\cite{M02}) gave another proof by using the Mellin-Barnes integral formula. 
Today, recurrence relations \eqref{eqn:recurrence relations with complex number sr} 
for MZFs are recognized as a fundamental tool for the study of MZFs.

%---  Theorem: reccurence relation for MZF  -------------------------------------------------------------------------------------%
\begin{cor}{(cf. \cite{AET01}, \cite{M02})}\label{thm: reccurence relation for MZF}
For $r\in\N_{\ge2}$, $l_r\in\N_0$, and $s_1,\dots,s_{r-1}\in\C$ except for singularities, we have
\begin{align}\label{eqn:recurrence relations with non-positive integer sr}
 &\zeta_r(s_1,\dots,s_{r-1},-l_r)\\
 &= -\frac{1}{l_r+1}\zeta_{r-1}(s_1,\dots,s_{r-1}-l_r-1) 
    + \sum_{k=0}^{l_r}\binom{l_r}{k}\zeta_{r-1}(s_1,\dots,s_{r-1}-l_r+k)\zeta(-k). \notag
\end{align}
\end{cor}
%---  Theorem: reccurence relation for MZF  -------------------------------------------------------------------------------------%

Here, we provide heuristic observation.
By Proposition \ref{prop: shuffle product formula for mzfs with zf of type Ar},
we have 
\begin{align}\label{eqn: shuffle for depth r-1 times depth 1 for reccurence relation}
 \zeta_{r-1}(s_1,\dots,s_{r-1})\zeta(s_r)
 &=\sum_{k=0}^{\infty} 
    \left[\binom{s_r+k-1}{k} \zeta_r(s_1,\dots,s_{r-1}-k,s_r+k) 
 \right. \\ &\qquad\left.
 - \binom{s_{r-1}+s_r+k-1}{s_{r-1}+k} \zeta_r(s_1,\dots,-k,s_{r-1}+s_r+k) \right]\notag \\
 &\quad +\sum_{k=0}^{\infty} 
      \left[ \binom{s_{r-1}+k-1}{k} \zeta(\bfs';s_r-k;s_{r-1}+k) \right.\notag\\
 &\qquad\quad\left. - \binom{s_{r-1}+s_r+k-1}{s_r+k} \zeta(\bfs';-k;s_{r-1}+s_r+k) \right],\notag
\end{align}
for $s_1,\dots,s_{r-1},s_r\in\C$ with $\re(s_i)\ge1$ ($i=1,\dots,r-2$), $\re(s_{r-1}),\,\re(s_r)>1$, 
$\bfs'=(s_1,\dots,s_{r-2})$ and for the second summation of the right-hand side, 
see Proposition \ref{prop: shuffle product formula for mzfs with zf of type Ar}.

\begin{assume}\label{assume: assumption}
We assume that \eqref{eqn: shuffle for depth r-1 times depth 1 for reccurence relation} holds when $s_r=-l$ ($l\in\N_0$).
\end{assume}

Under this assumption, if we put $s_r=-l_r\in\Z_{\le0}$, we have
\begin{align*}
 \binom{-l_r+k-1}{k} 
 = \begin{cases}
    (-1)^k\binom{l_r}{k} \quad&(k=0,\dots,l_r) \\
    0 \quad&(k\ge l_r+1),
   \end{cases}
\end{align*}
$\binom{s_{r-1}-l_r+k-1}{s_{r-1}+k}=0$ for $k\in\N_0$ and $\binom{s_{r-1}-l_r+k-1}{-l_r+k}=0$ for $0\le k\le l_r-1$.
Thus the first summation of the right-hand side of \eqref{eqn: shuffle for depth r-1 times depth 1 for reccurence relation}
remains only $k=0,\dots,l_r$ and $k=l_r+1$.
Because, for $k=l_r+1$, the multiple zeta function $\zeta_r(s_1,\dots,s_{r-1}-k,-l_r+k)$ 
has a simple pole and this cancels out with a zero of the function $\frac{1}{\Gamma(s_r)}$ at $s_r=-l_r$.
The second summation of the right-hand side of \eqref{eqn: shuffle for depth r-1 times depth 1 for reccurence relation} vanishes. 
Therefore, when $s_r=-l_r$, \eqref{eqn: shuffle for depth r-1 times depth 1 for reccurence relation} can be written as follows;
\begin{align*}% \label{eqn: reccurence relation before binomial transformation}
 &\zeta_{r-1}(s_1,\dots,s_{r-1})\zeta(-l_r) \\
 &= \sum_{k=0}^{l_r}(-1)^{k}\binom{l_r}{k} \zeta_r(s_1,\dots,s_{r-1}-k,-l_r+k) 
    + \frac{(-1)^{l_r}}{l_r+1}\zeta_{r-1}(s_1,\dots,s_{r-1}-l_r-1).%\notag
\end{align*}
By Lemma \ref{lem: binomial transformation}, we get the equation \eqref{eqn:recurrence relations with non-positive integer sr}.
\qed
%---  Proof of theorem: reccurence relation for MZFs  ---------------------------------------------------------------------------%

%---  Lemma: binomial transformation  -------------------------------------------------------------------------------------------%
\begin{lem}\label{lem: binomial transformation}
Let $f,\,g:\C\times\Z\rightarrow\C$ be maps. We assume that 
\begin{align*}
 g(s;-l) = \sum_{k=0}^l \binom{l}{k}f(s-k;-l+k)
\end{align*}
for $s\in\C$ and $l\in\N_0$. Then we have 
\begin{align*}
 f(s;-l) = \sum_{k=0}^l (-1)^k\binom{l}{k}g(s-k;-l+k).
\end{align*}
\end{lem}
%---  Lemma: binomial transformation  -------------------------------------------------------------------------------------------%

Under Assumption \ref{assume: assumption}, the recurrence relations \eqref{eqn:recurrence relations with non-positive integer sr} and shuffle product relations for MZFs are equivalent in the special case.  
We point out that the recurrence relations \eqref{eqn:recurrence relations with non-positive integer sr} for MZFs are useful for the calculation of values of MZFs at integer points.
Indeed, the second-named author (\cite{Sh24}) proved that the values of MZFs at regular integer points can be expressed 
as a finite linear combination of MZVs by using generalized recurrence relations.
He proved the generalized recurrence relations for MZFs by using original recurrence relations \eqref{eqn:recurrence relations with non-positive integer sr} and the stuffle product relations for MZFs.
In other words, the generalized recurrence relations for MZFs could be regarded as 
the special cases of functional double shuffle relations for MZFs.

We remark that the recurrence relations of MZFs have more generalized form (cf. \cite{M02});
\begin{align}\label{eqn:recurrence relations with complex number sr}
  \zeta_r(s_1,\dots,s_{r-1},s_r)
  &= \frac{1}{s_r-1}\zeta_{r-1}(s_1,\dots,s_{r-1}+s_r-1) \\
  &\quad+ \sum_{k=0}^{M}\binom{-s_r}{k}\zeta_{r-1}(s_1,\dots,s_{r-1}+s_r+k)\zeta(-k) + I(\bfs,M,\epsilon) \notag
\end{align}
for $M\in\N$, $\epsilon>0$, and 
\begin{align*}
 I(\bfs,M,\epsilon)
 :=\frac{1}{2\pi i}\int_{(M-\epsilon)}\frac{\Gamma(s_r+z)\Gamma(-z)}{\Gamma(s_r)}\zeta_{r-1}(s_1,\dots,s_{r-1}+s_r+z)\zeta(-z)dz,
\end{align*}
where $(M-\epsilon)$ is the path of integration on the vertical line 
$\re(z)=M-\epsilon$ from $M-\epsilon-i\infty$ to $M-\epsilon+i\infty$.
We also know that the shuffle product relation for $\zeta_{r-1}(s_1,\dots,s_{r-1})\zeta(s_r)$. 
These observations lead us to the following problems;

\vspace{3mm}
\noindent
{\bf Problem. }
\begin{enumerate}
 \item Do the shuffle product relations for MZFs hold for wider region?
 \item Are the shuffle product relations and recurrence relations for MZFs equivalent in general? 
\end{enumerate}

If the second problem is positive, then we can say that the shuffle product relations for MZVs/MZFs have an analytic interpretation.
Then, is there any analytic interpretation of stuffle product relations?
%---  subsection: Recurrence relations for MZFs  --------------------------------------------------------------------------------%
%===  SECTION END  ===============================================================================================================%

%---  Shuffle product for multiple zeta functions  -------------------------------------------------------------------------------%
%---  References  -------------------------------------------------------------------------------------------------------%

% 関数関係式をやってる人を列挙すべきか。
% Matsuoka, Ikeda, これについて言及するのは難しい。一般に調和積以外の関係式がなさそうなことを述べているが、彼らの結果について触れるのはちとめんどい。
% root系の関係式について触れるのは良い。

%---  Shuffle product for multiple zeta functions  -------------------------------------------------------------------------------%
%---  Template for notations  ---%
%---  Hypergeometric function 2F1:  {}_2F_1\left(\begin{matrix} \alpha,\, \beta \\ \gamma \end{matrix}\,\middle|\, z \right)
%---  Shuffle product for multiple zeta functions  ------------------------------------------------------------------------------%

\begin{thebibliography}{99}
\bibitem[AET]{AET01} S. Akiyama, S. Egami and Y. Tanigawa,
{\it Analytic continuation of multiple zeta-functions and their values at non-positive integers,}
Acta Arith., {\bf 98}, 2001. 

%\bibitem{AK99} T. Arakawa and M. Kaneko,
%{\it Multiple zeta values, poly-Bernoulli numbers, and related zeta functions}
%Nagoya Math. J., {\bf 153}, 1999.

%\bibitem[AK]{AK05} T. Arakawa and M. Kaneko,
%{\it Notes on Multiple zeta values and multiple $L$ values (in Japanese),}
%Lecture Note, Rikkyo Univ., 2005.
\bibitem[AK]{AK-MLVs}T. Arakawa, M. Kaneko,
{\it On multiple L-values,}
J. Math. Soc. Japan 56 (2004), no. {\bf 4}, 967--991.

\bibitem[BZ]{BZ10} D. M. Bradley and X. Zhou,
{\it On Mordell-Tornheim sums and multiple zeta values,}
Ann. Sci. Math. Qu\'{e}bec, {\bf 34} 2010.

%\bibitem[E]{E1776} L. Euler,
%{\it Meditationes circa singlare serierum genus,}
%Novi Comm. Acad. Sci. Petropol. {\bf 20}, 1776 
%(reprinted in Opera Omnia, ser. I, vol. 15, B. G. Teubner, Berlin, 1927.).

%\bibitem[EMS]{EMS} K. Ebrahimi-Fard, D. Manchon, J. Singer,
%{\it The Hopf algebra of (q)multiple polylogarithms with non-positive arguments},
%Int. Math. Res. Notices, 2017, Vol. {\bf 16}, 4882--4922.

%\bibitem[EM21]{EM21} D. Essouabri and K. Matsumoto,
%{\it Values of general multiple zeta-functions with polynomial denominators at non-positive integer points,}
%Int. J. Number Theory, {\bf 32}, 2021.

\bibitem[HMO18]{HMO18} M. Hirose, H. Murahara and T. Onozuka,
{\it Sum formula for multiple zeta function,}
arxiv:1808.01559v2.

\bibitem[HMO20]{HMO20} M. Hirose, H. Murahara and T. Onozuka,
{\it An interpolation of Ohno's relation to complex functions,}
Math. Scand, {\bf 126} 2020.

%\bibitem{HWZ96} J. G. Huard, K. S. Williams and N.-Y. Zhang, 
%{\it On Tornheim's double series,}
%Acta Arith. {\bf 75}, 1996.

\bibitem[IKZ]{IKZ06} K. Ihara, M. Kaneko and D. Zagier,
{\it Derivation and double shuffle relations for multiple zeta values,}
Compositio Math., {\bf 142}, 2006.

\bibitem[J]{J10} S. Joyner,
{\it On a generalization of Chen's iterated integrals,}    
J. Number Theory, {\bf 130}, 2010.

\bibitem[KMT11]{KMT11} Y. Komori, K. Matsumoto and H. Tsumura,
{\it Shuffle products for multiple zeta values and partial fraction decompositions of zeta-functions of root systems,} 
Math. Z., {\bf 268}, 2011.

\bibitem[KMT24]{KMT24} Y. Komori, K. Matsumoto and H. Tsumura,
{\it The Theory of Zeta-Functions of Root Systems,}
Springer Nature Singapore, 2024.  

\bibitem[KS]{KS23} N. Komiyama and T. Shinohara,
{\it Shuffle product of desingularized multiple zeta functions at integer points,} 
arXiv:2302.11444.

\bibitem[M02]{M02} K. Matsumoto,
{\it On the analytic continuation of various multiple zeta-functions,}
Number Theory for the Millenium (Urbana, 2000), {\bf 11}, M. A. Bennett et. al. (eds.), 
A. K. Peters, Natick, MA, 2002.

%\bibitem{M04} K. Matsumoto,
%{\it Functional equations for double zeta-functions,}
%Math. Proc. Cambridge Phil. Soc. {\bf 136}, 2004.
% これについてはどう言及しようかしら。

\bibitem[M06]{M06} K. Matsumoto, 
{\it Analytic properties of multiple zeta-functions in several variables,} 
in "Number Theory: Tradition and Modernization", W. Zhang and Y. Tanigawa (eds.), Springer, New York, 2006.

\bibitem[MO]{MO22} H. Murahara and T. Onozuka,
{\it Cyclic relation for multiple zeta functions,}
Res. Number Theory, {\bf 8}, 2022.

%\bibitem{MT15} K. Matsumoto and H. Tsumura,
%{\it Mean value theorems for the double zeta-function,}
%J. Math. Soc. Japan, {\bf67}, 2015.

\bibitem[N]{N06} T. Nakamura,  
{\it A functional relation for the Tornheim double zeta function,}
Acta Arith. {\bf 125}, 2006. 
    
\bibitem[T]{T07} H. Tsumura, 
{\it On functional relations between the Mordell-Tornheim double zeta functions and the Riemann zeta function,}
Math. Proc. Cambridge Philos. Soc., {\bf 142}, 2007.

%\bibitem{Sa22} B. Saha,
%{\it Multiple Stieltjes constants and Laurent type expansion of the multiple zeta functions at integer point,}
%Selecta Math. (N.S.), {\bf 28} 2022.

\bibitem[S]{Sh24} T. Shinohara,
{\it Multiple zeta functions at regular integer points,}
Acta Arith., {\bf 212}, 2024.

\bibitem[TE]{TE} F. G. Tricomi and A. Erd\'elyi, 
{\it The asymptotic expansion of a ratio of gamma functions,}
Pacific J. Math. {\bf 1}, 1951.

\bibitem[WW]{WW27} E. T. Whittaker and G. N. Watson,
{\it A course of modern analysis,}
4th ed., Cambridge Univ. Press, 1927.

%\bibitem[XZ]{XZ} H. Xiang, B. Zhang,
%{\it Extended Shuffle Product for Multiple Zeta Values},
%arXiv:2411.08536.

\bibitem[Z]{Z16} J. Zhao,
{\it Multiple zeta functions, multiple polylogarithms and their special values,} 
Series on Number Theory and its Applications, 12. World Scientific Publishing Co. Pte. Ltd., Hacken- sack, NJ, 2016.
\end{thebibliography}
\end{document}